\documentclass[11pt]{article}

\usepackage[matrix,arrow,arc]{xy}
\usepackage{enumerate}
\usepackage{amssymb}

\usepackage{color}

\newcommand{\vt}{\vspace{.3cm}}

\newcommand{\Cc}{{\mathbb C}}

\newcommand{\ka}{{\kappa}}

\newcommand{\benu}{\begin{enumerate}}
\newcommand{\enu}{\end{enumerate}}

\newcommand{\beqna}{\begin{eqnarray}}
\newcommand{\eqna}{\end{eqnarray}}
\newcommand{\beqnast}{\begin{eqnarray*}}
\newcommand{\eqnast}{\end{eqnarray*}}
\newcommand{\beqn}{\begin{equation}}
\newcommand{\eqn}{\end{equation}}
\newcommand{\beqnst}{\begin{equation*}}
\newcommand{\eqnst}{\end{equation*}}

\newcommand{\wtil}{\tilde{w}}

\newcommand{\um }{\mathbf{1}_A}

\newcommand{\Hom}{\operatorname{Hom}}

\newcommand{\ot}{\otimes}

\newcommand{\m}{{}^{-1}}

\newcommand{\bu}{\bullet}

\usepackage{amsmath,amsthm,amssymb}
\usepackage{amsxtra,amscd, amsbsy}
\usepackage{eucal}

\newtheorem{thm}{Theorem}[section]

\newtheorem{defi}[thm]{Definition}

\theoremstyle{definition}
\newtheorem{ex}[thm]{Example}
\newtheorem{remark}[thm]{Remark}

\usepackage{enumerate}

\begin{document}

\begin{center}
\Large{\bf Globalization of Twisted partial Hopf actions}\footnote[0]
{The third and the fourth authors were partially supported by Fapesp of Brazil. The first and third authors were also partially supported by CNPq of Brazil.}

\end{center}

\begin{center}
{\bf Marcelo M. S. Alves$^\dag$, Eliezer Batista$^\ddag$, Michael Dokuchaev$^\triangledown$ and Antonio Paques$^\blacktriangledown$}\\

\vt

{ \footnotesize

$\dag$ Departamento de Matem\'atica\\
Universidade Federal do Paran\'a\\
81531-980, Curitiba, PR, Brazil\\
E-mail: {\it marcelomsa@ufpr.br}\\

$^\ddag$ Departamento de Matem\'atica\\ 
Universidade Federal de Santa Catarina\\
88040-900, Florian\'opolis, SC, Brazil\\
E-mail: {\it eliezer1968@gmail.com}\\ 

$^\triangledown$ Instituto de Matem\'atica e Estat\'istica\\
Universidade de S\~ao Paulo\\
05508-090, S\~ao Paulo, SP, Brazil \\
E-mail: {\it dokucha@ime.usp.br}\\

$^\blacktriangledown$Instituto de Matem\'atica\\
Universidade Federal do Rio Grande do Sul\\
91509-900, Porto Alegre, RS, Brazil\\
E-mail: {\it paques@mat.ufrgs.br}}
\end{center}

\begin{abstract} In this work, we review some properties of twisted partial actions of Hopf algebras on unital algebras and give  necessary and sufficient conditions for a twisted partial action to have a globalization. We also  elaborate a series of examples.
\end{abstract}

\section{Introduction}

Partial group actions arose in  operator algebra theory in  \cite{E-1}, \cite{Mc}, \cite{E0}, \cite{E1}, in order to characterize  $C^*$-algebras generated by partial isometries as more  general crossed product. The algebraic study of partial actions and related notions  was initiated in 
 \cite{E1}, \cite{DEP}, \cite{St1}, \cite{St2}, \cite{KL} and \cite{DE}, motivating further investigations.  In particular, the Galois theory of partial group actions developed in \cite{DFP} inspired  further Galois theoretic results in \cite{BP}, \cite{CaenDGr},   \cite{FrP},  \cite{KuoSzeto}, \cite{PRSantA},  as well as the  introduction and study of partial Hopf actions and coactions in \cite{CJ}, which, in turn,   became  the starting point for further research in partial Hopf (co)actions in \cite{AAB},  \cite{AB2}, \cite{AB},  \cite{AB3}, \cite{ABDP} and \cite{ABV}. 
 
 The concepts of a twisted partial action and the corresponding 
crossed product, introduced  for $C^*$-algebras in \cite{E0}, and adapted for abstract rings in \cite{DES},  suggested the idea of extending these notions for Hopf algebras, unifying  twisted partial group actions, partial Hopf actions and twisted actions of Hopf algebras. This was done in \cite{ABDP},  where such crossed products were related to the so-called  partially cleft extensions of algebras, and   examples were elaborated using algebraic groups.

 One obtains partial (group or Hopf)  actions by restricting global ones to (non-necessarily) invariant  ideals, and it is a central problem to discover under which conditions a given partial action is globalizable, i.e. it can be obtained (up to isomorphism) as a restriction of a global action.

  The study of the globalization problem of partial group actions  was initiated  in  the PhD Thesis by F. Abadie \cite{Abadie} (see also  \cite{AbadieTwo}), and independently by B. 
  Steinberg in  \cite{St2}, and J. Kellendonk and M. V. Lawson in \cite{KL}. Other globalization results were obtained in \cite{AB},  \cite{AB3},  \cite{BCFP}-\cite{BeF},  \cite{CF2}, \cite{CFMarcos}, \cite{DE}, \cite{DES2}, \cite{DdRS}, \cite{EGG}, \cite{F}, \cite{Gil}, \cite{Megre2}.  In particular, in \cite{AB2} a glo\-ba\-li\-za\-tion theorem for  the case of partial actions of Hopf algebras on unital algebras was proved, whereas in \cite{DES2} the  globalization of twisted partial group actions on rings was investigated. The importance of the  glo\-ba\-li\-za\-tion problem lies in the possibility to relate partial actions  with  global ones and try to use known results on global actions to obtain more general facts.

 The aim of this article is to go further on the investigation initiated in \cite{ABDP} by considering  the globalization of  twisted partial Hopf actions. In Section~\ref{SymTwParAc}, we recall some basic concepts about  twisted partial actions of Hopf  algebras, which we formalize now in the form which is most convenient for our purpose.   We start with the concept of a partial measuring map of a bialgebra (or Hopf algebra) $H$ on a unital algebra $A$, giving, in our examples, a complete characterization of these maps in the case of $H$ being the group algebra $\kappa G$, the dual $(\kappa G)^*$ of the group algebra  $ \kappa G $ of a finite group $G,$ and the Sweedler Hopf algebra $H_4$, and $A$ coinciding (in all three cases) with  the base field $\kappa $. Then we proceed with the concept of a twisted partial action of a Hopf algebra (or, more generally, a bialgebra),  in which we include all properties needed to guarantee that the partial crossed product is  associative and unital. Next we consider the symmetric twisted partial Hopf actions. Symmetric twisted partial actions are described  in details for specific Hopf algebras. In particular,  the case of a group algebra $\kappa G$ recovers  the theory of twisted partial actions of groups as developed in \cite{DES} and \cite{DES2}. For the Sweedler Hopf algebra $H_4$, the only symmetric twisted partial actions are the global ones. Furthermore,  the dual of the group algebra $(\kappa G)^*$ of a finite group $G,$ the partial  cocycles appearing have remarkable symmetries, and we relate them to  global cocycles of  dual group algebras of quotient groups.  We do thoroughly the specific example for the Klein four-group $K_4$, acting on the base field $\kappa ,$ and show that the symmetric twisted partial actions of $(\kappa K_4)^{\ast}$ on $\kappa $ are parametrized by the zeros
 $(x, y) \in \kappa ^2$ of a polynomial in $x, y$ of degree $2.$ 

 Section 3 is dedicated to the globalization theorem itself. The precise meaning of globalization of a symmetric twisted partial action of a Hopf algebra $H$ on a unital algebra $A$ given by a symmetric pair of partial cocycles $\omega$ and $\omega '$ is the existence of an algebra $B$ (not necessarily unital) such that there is a twisted (global) action of $H$ on $B$ with convolution invertible cocycle $v:H\otimes H \rightarrow B$ and an algebra monomorphism $\varphi :A\rightarrow B$ such that its image $\varphi (A)$ is a unital ideal in $B,$ and the restriction of this global twisted action is isomorphic to the original twisted partial action on $A$. The main result can be formulated as follows:  a symmetric twisted partial action of a Hopf algebra $H$ on a unital algebra $A$ associated to the symmetric pair of partial cocycles $\omega$ and $\omega '$,  is globalizable if, and only if, there exists a normalized convolution invertible linear map $\tilde{\omega} :H\otimes H\rightarrow A$ satisfying certain compatibility conditions for intertwining the partial action of $H$ on $A$ and the restriction of the twisted action of $H$ on $B$. The question for uniqueness of a globalization is not so straightforward to be addressed in this context as it is for partial actions of groups \cite{DE} and partial actions of Hopf algebras \cite{AB}. Even the form of the algebra $B$ cannot be given in a simple way for the general case. We conclude the section with two examples. The first one shows that the example of a twisted partial Hopf action, constructed in \cite{ABDP} using the relation between algebraic groups and Hopf algebras, is globalizable. The second one is an explicit partial cocycle for   $(\kappa K_4)^{\ast}$ which leads to a globalizable symmetric twisted partial action.
 
 \section{Symmetric twisted partial actions}\label{SymTwParAc}

In  this paper, unless otherwise explicitly stated,  $\kappa $ will denote  an  arbitrary (associative)  unital commutative ring and unadorned $\otimes$ will mean $\otimes_{\kappa }$. We also omit the sum in the Sweedler notation for the comultiplication of a Hopf algebra. {In what follows the symbol $1$ will stand for the unit element $1_{ \kappa }$  of $\kappa .$ }

\begin{defi} \label{defi:measure} Let $H$ be a $\kappa $-bialgebra, $A$ a unital $\kappa $-algebra with unit element $\um $. We say that $H$ partially measures $A$ if there is a linear map 
\[
\begin{array}{rccc} \cdot : & H\otimes A & \to & A ,\\
\, & h\otimes a & \mapsto &h\cdot a,
\end{array}
\]
satisfying the following conditions:
\begin{enumerate}
\item[(PM1)] $1_H\cdot a=a$, for every $a\in A$, \label{unitpartial}
\item[(PM2)] $h\cdot (ab)=(h_{(1)}\cdot a)(h_{(2)}\cdot b)$, for every $h\in H$ and $a,b\in A$, \label{productpartial}
\item[(PM3)] $h\cdot (k\cdot \um )=(h_{(1)} \cdot \um )(h_{(2)}k\cdot \um )$, for every $h,k\in H$. \label{associativtyone}
\end{enumerate}
The map $\cdot :H\otimes A \rightarrow A$ is called a measuring map of $H$ on $A$.
\end{defi}

Note that, if $h\cdot \um =\epsilon (h) \um$, then we have the usual notion that $H$ measures $A$.

{ 

\begin{ex} Note that every partial Hopf action in the sense defined by Caenepeel and Janssen \cite{CJ} is an example of a measuring map of a Hopf algebra $H$ on a unital algebra $A$.  Indeed, if  $H$ is a Hopf algebra then a partial action of $H$ on a  unital $\kappa $-algebra $A$  is a linear map 
\[
\begin{array}{rccc}  & H\otimes A & \to & A ,\\
\, & h\otimes a & \mapsto &h\cdot a,
\end{array}
\] which satisfies (PM1), (PM2) and   

\begin{equation}  \label{parAction}
h\cdot (k\cdot a )=(h_{(1)} \cdot \um )(h_{(2)}k\cdot a ), \;\;\text{for every}\;\; h,k\in H, a\in A.
\end{equation}

\end{ex} 

}

It is readily seen that a measuring of a Hopf algebra on the base field is automatically a partial action.  In what follows we revisit three examples of such measurings that appeared previously in \cite{AAB} and \cite{ABV}: We provide more detailed versions of \cite[Ex. 3.7]{AAB} and \cite[Prop. 6.2]{ABV} in Examples \ref{exemplo:grupo} and \ref{exemplo:graduacao} respectively, and 
in Example \ref{exemplo:sweedler} 
we redo the calculations of  \cite[Ex. 3.8]{AAB}, regarding partial actions of the Sweedler algebra on $\kappa$, in a simpler and more readable manner.

\begin{ex} \label{exemplo:grupo}  Let $G$ be a group and $\kappa $ a field. One can classify all the partial measuring maps of the Hopf algebra $\kappa G$ over the base field $\kappa$, which correspond to the partial actions of $G$ on $\kappa .$ The map $\cdot :\kappa G \otimes \kappa \rightarrow \kappa$ can be viewed as a linear functional $\lambda :\kappa G \rightarrow \kappa$ given by $\lambda_g =\lambda (g)=g\cdot 1 $. 
The condition (PM1) implies that $\lambda (e) = 1$ (where $e$ denotes $1_G$) and the condition (PM2) means that $\lambda_g =\lambda_g \lambda_g$, which implies that either $\lambda_g =1$ or $\lambda_g =0$. Consider now the following subset of the group $G$,
\[
L= \{ g\in G \; | \; \lambda_g =1 \} .
\]
By item (PM), we have that $\lambda_g \lambda_h =\lambda_g \lambda_{gh}$, therefore, if $g,h\in L$ then $gh \in L$. In the same way, putting $h=g^{-1}$, one can conclude that if $g\in L$ then $g^{-1} \in L$, and it follows that $L$ is a subgroup of $G$. Therefore, the partial measurings of the Hopf algebra $\kappa G$ over the base field $\kappa$ are classified by the subgroups of $G$.
\end{ex}

\begin{ex} \label{exemplo:graduacao} Let $G$ be a finite group.
We can classify all the partial measuring maps of the Hopf algebra $(\kappa G)^*$ over a field $\kappa .$ 
By \cite{ABV} they correspond to the partial $G$-gradings of $\kappa .$ Again, they can be classified by linear functionals $\lambda :(\kappa G)^* \rightarrow \kappa$, defined by $\lambda (p_g )=p_g \cdot 1$, where $\{ p_g \}_{g\in G}$ is the canonical  basis of  $(\kappa G)^*$. 
We know, from item (PM1), that $\sum_{g\in G} \lambda (p_g )=1$, then there are some $g\in G$ such that $\lambda (p_g )\neq 0$. Again, consider the subset
\[
L= \{ g\in G \; | \; \lambda (p_g ) \neq 0 \} .
\]
The condition (PM3) reads
\[
\lambda (p_g ) \lambda (p_h ) = \lambda (p_{gh^{-1}}) \lambda (p_h ). 
\]
This implies that, if $g,h\in L$ then $gh^{-1} \in L$, therefore $L$ is a subgroup of $G$. Moreover, putting $h=g$, we obtain $\lambda (p_g )= \lambda (p_e )$, for every $g\in L$. Finally, by (PM1), we have
\[
\sum_{g\in G} \lambda (p_g )=\sum_{g\in L} \lambda (p_e )= |L| \lambda (p_e) =1 ,
\] 
which implies that ${\rm char}\; \kappa \nmid |L|$ and 

\[ \lambda (p_g )=\lambda (p_e)=\frac{1}{|L|} , \; \forall g\in L.
\]
 
\noindent Therefore, the partial $G$-gradings of the base field $\kappa$ are classified by the subgroups $L\leq G,$  whose order is not divisible by ${\rm char} \; \kappa $    and the linear functionals $\lambda^L : (\kappa G)^* \rightarrow \kappa$ given by 
\[
\lambda^L (p_g )=\left\{ \begin{array}{lcr} \frac{1}{|L|} & \mbox{if} & g\in L \\ 0 & \mbox{if} & g\not\in L \end{array} \right. .
\]
\end{ex}  

\begin{ex} \label{exemplo:sweedler} Consider the four dimensional Sweedler Hopf algebra $H_4 =H=\langle 1_H,x,g \; | \; x^2 =0, \; g^2 =1_H , \; gx=-xg \rangle$ over a field $\kappa ,$ whose characteristic is not $2,$ with the comultiplication and the counit given by
\[
\Delta (g) =g\otimes g , \quad \epsilon (g)=1, \quad \Delta (x) = x\otimes 1_H +g\otimes x , \quad \epsilon (x )=0 .
\]
Let us classify the partial measuring maps of $H_4$ on the base field $\kappa$. Each partial measuring map is given by a functional $\lambda :H_4 \rightarrow \kappa$ defined as $\lambda_h =\lambda (h) =h\cdot 1$. By (PM1), we have that $\lambda_{1_H} =1$ and by (PM2) we conclude that $\lambda_g =\lambda_g \lambda_g$, then $\lambda_g =1$ or $\lambda_g =0$. Assuming $\lambda_g =1$, we  have from (PM3)
\begin{eqnarray*}
\lambda_g \lambda_x & = & \lambda_g \lambda_{gx} ,  \\
\lambda_x \lambda_g & = & \lambda_x \lambda_{g} +\lambda_g \lambda_{xg} ,  \\
\lambda_{xg} \lambda_g & = & \lambda_{xg} \lambda_{1_H} +\lambda_{1_H} \lambda_x.
\end{eqnarray*}
The first equality implies that $\lambda_x =\lambda_{gx}=-\lambda_{xg}$, and the second leads to $\lambda_x=0$. Therefore, for $\lambda_g =1$, the only solution is the global case $\lambda =\epsilon$. Considering instead $\lambda_g =0$ we obtain from the above equations  again that $\lambda_x =\lambda_{gx} =-\lambda_{xg}$ and, in fact,  no other constraint appears. Therefore, for $\lambda_g =0$ one can assign any value for $\lambda_x \in \kappa$, and each such functional $\lambda$ gives rise to a partial measuring of $H_4$ on $\kappa$. 
\end{ex}

 The definition of a twisted partial action of a bialgebra (or Hopf algebra) $H$ on a unital algebra $A$   was given in  \cite[Definition 1]{ABDP} in a rather general form. We shall use the following version which already incorporates the normalization condition \cite[(16)]{ABDP}, the $2$-cocycle equality    \cite[(17)]{ABDP} and also (iii) of \cite[Definition 2]{ABDP}.

\begin{defi} \label{defi:twisted} A twisted partial action of a bialgebra (or Hopf algebra) $H$ on a unital algebra $A$ consists of two linear maps, $\cdot :H\otimes A\to A$ and $\omega:H\otimes H\to A$ satisfying the following conditions:
\begin{enumerate}
\item[(TPA1)] $H$ partially measures $A$ via the linear map $\cdot$. 
\item[(TPA2)] $(h_{(1)}\cdot (l_{(1)}\cdot a))\omega(h_{(2)} ,l_{(2)}) = \omega(h_{(1)} ,l_{(1)})(h_{(2)}l_{(2)}\cdot a)$, for every $h,l\in H$ and $a\in A$. 
\item[(TPA3)] $\omega (h ,l) =  \omega (h_{(1)} , l_{(1)})(h_{(2)}l_{(2)}\cdot \um )$, for every $h,l\in H$. 
\item[(TPA4)] $\omega (h, 1_H )=\omega (1_H, h)=h\cdot \um$, for every $h\in H$. 
\item[(TPA5)] $(h_{(1)}\cdot \omega (k_{(1)} , l_{(1)}))\omega (h_{(2)}, k_{(2)}l_{(2)})=\omega(h_{(1)} ,k_{(1)})\omega (h_{(2)}k_{(2)} ,l)$, for every $h,k,l\in H$.
\end{enumerate}
The algebra $A$, in its turn, is called a twisted partial $H$-module algebra.
\end{defi}

The map $\omega$ is called a partial 2-cocycle or twisting.
  If $h\cdot \um =\epsilon (h)\um$, then a partial measuring is a usual one and a twisted partial action is merely a twisted action in the usual sense. On the other hand, if the partial 2-cocycle satisfies $\omega (h,k)=h\cdot (k\cdot \um )$, then it is called a trivial partial cocycle and the resulting twisted partial action is the usual partial action, as defined in \cite{CJ}.
  
Note that by \cite[Prop. 1]{ABDP}, Definition~\ref{defi:twisted}  implies 
  
  \begin{equation}\label{Prop1Israel}
   \omega (h,l)= (h_{(1)}\cdot (l_{(1)}\cdot \um ))\omega(h_{(2)} ,l_{(2)}) = (h_{(1)}\cdot  \um)   \omega(h_{(2)} ,l),
  \end{equation} for every $h,l\in H.$ 

\begin{ex}\label{restriction} A source of examples of twisted partial actions can be given by the restriction to a unital ideal of a twisted action. More precisely, let $H$ be a Hopf algebra with invertible antipode and $B$ a, possibly non-unital, algebra measured by $H$ through the linear map $\triangleright :H\otimes B \rightarrow B$. In the non-unital case, we simply consider the condition 
\[
h\triangleright (ab)=(h_{(1)} \triangleright a)(h_{(2)}\triangleright b) ,
\]
since there is no meaning for the expression $h\triangleright 1_B =\epsilon (h)1_B$.
A twisted action of $H$ on $B$ is given by this measuring and a convolution invertible linear map $u:H\otimes H \rightarrow M(B)$, where $M(B)$ is the multiplier algebra of  
$B$, such that 
\begin{eqnarray*}
u(h, 1_H )=u(1_H, h) & = & \epsilon (h) 1_{M(B)} ,\\
(h_{(1)}\triangleright (l_{(1)}\triangleright a)) u(h_{(2)} ,l_{(2)}) & = & u (h_{(1)} ,l_{(1)}) (h_{(2)}l_{(2)}\triangleright a) ,\\
(h_{(1)}\overline{\triangleright} u (k_{(1)} , l_{(1)}))u (h_{(2)}, k_{(2)}l_{(2)}) & = & u(h_{(1)} ,k_{(1)}) u(h_{(2)}k_{(2)} ,l) .
\end{eqnarray*}
Note that the first and the third equalities above are identities of multipliers, $1_{M(B)}$ is simply the identity map in $B$. The map $\overline{\triangleright}:H\otimes M(B) \rightarrow M(B)$ which appears in the third identity requires a more detailed explanation. First, the multiplier algebra $M(B)$ can be viewed as a pull back \cite{JV}
\[
\xymatrix{M(B) \ar[r]^-{\pi_1} \ar[d]_-{\pi_2 }  & \mbox{End}_B (B)  \ar[d]^-{\lambda} \\
{}_B\mbox{End} (B) \ar[r]_-{\rho} & {}_B\mbox{Hom}_B (B\otimes B , B)} 
\]
where the linear map $\lambda :\mbox{End}_B (B)\rightarrow {}_B \mbox{Hom}_B (B\otimes B ,B)$ is given by $\lambda (L)(a\otimes b)=aL(b)$, the map $\rho :{}_B \mbox{End} (B) \rightarrow  {}_B \mbox{Hom}_B (B\otimes B ,B)$ is given by $\rho (R) (a\otimes b)=R(a)b$, and $\pi_1$ and $\pi_2$ are the canonical projections. That is equivalent to the classical definition of the multiplier algebra\footnote{The concept of a multiplier algebra first appeared in the context of C*-algebras \cite{APT}. For abstract algebras its detailed definition can be seen, for example, in \cite{DE}.} as the algebra generated by pairs $(L,R)\in \mbox{End}_B (B) \times {}_B \mbox{End} (B)$ satisfying the identity $aL(b)=R(a)b$. For simplicity, let us denote a multiplier $x=(L,R)$ such that $L(a)=xa$ and $R(a)=ax$.  Then, we can use the universal property of a pull back in order to define
and action $\overline{\triangleright} :H\otimes M(B)\rightarrow M(B)$. First, define two linear maps $\Phi :H\otimes M(B)\rightarrow \mbox{End}_B (B)$ and $\Psi :H\otimes M(B)\rightarrow {}_B\mbox{End} (B)$ respectively by
\begin{eqnarray*}
\Phi (h\otimes x)(b) & = & h_{(1)}\triangleright (x(S(h_{(2)})\triangleright b)) \\
\Psi (h\otimes x)(b) & = & h_{(2)}\triangleright 
((S^{-1}(h_{(1)})\triangleright b)x).
\end{eqnarray*}
It is easy to see that $\Phi (h\otimes x)$ really belongs to $\mbox{End}_B (B)$, indeed
\begin{eqnarray*}
\Phi (h\otimes x )(ab) & = & h_{(1)}\triangleright (x(S(h_{(2)})\triangleright (ab))) \\
& = & h_{(1)}\triangleright (x((S(h_{(3)})\triangleright a)(S(h_{(2)})\triangleright b))) \\
& = & h_{(1)}\triangleright ((x (S(h_{(3)})\triangleright a))(S(h_{(2)})\triangleright b)) \\
& = & (h_{(1)}\triangleright (x (S(h_{(4)})\triangleright a))) (h_{(2)}\triangleright (S(h_{(3)})\triangleright b)) \\
& = & (h_{(1)}\triangleright (x (S(h_{(4)})\triangleright a))) (h_{(2)}S(h_{(3)})\triangleright b) \\
& = & (h_{(1)}\triangleright (x (S(h_{(2)})\triangleright a))) b \\
& = & (\Phi (h\otimes x)(a))b.
\end{eqnarray*}
Similarly, one proves that $\Psi (h\otimes x) \in {}_B\mbox{End} (B)$. Finally, $\lambda \circ \Phi =\rho \circ \Psi$, indeed, consider $a,b\in B$ $h\in H$ and $x\in M(B)$, then
\begin{eqnarray*}
\lambda \circ \Phi (h\otimes x)(a\otimes b) & = & a(h_{(1)}\triangleright (x (S(h_{(2)})\triangleright b))) \\
& = & (h_{(2)}S^{-1} (h_{(1)})\triangleright a)(h_{(3)}\triangleright (x (S(h_{(4)})\triangleright b)))\\
& = & h_{(2)}\triangleright ((S^{-1} (h_{(1)})\triangleright a)(x (S(h_{(3)})\triangleright b)))\\
& = & h_{(2)}\triangleright (((S^{-1} (h_{(1)})\triangleright a)x)(S(h_{(3)})\triangleright b))\\
& = & (h_{(2)}\triangleright ((S^{-1} (h_{(1)})\triangleright a) x)) (h_{(3)}S(h_{(4)})\triangleright b)\\
& = & (h_{(2)}\triangleright ((S^{-1} (h_{(1)})\triangleright a) x))b\\
& = & \rho \circ \Psi (h\otimes x)(a\otimes b) .
\end{eqnarray*}
Therefore, by the universal property of $M(B)$ as a pull back, there is a unique linear map $\overline{\triangleright}:H\otimes M(B)\rightarrow M(B)$ given by
\begin{eqnarray*}
(h\overline{\triangleright} x)b & = & h_{(1)}\triangleright 
(x (S(h_{(2)})\triangleright b)) \\
b (h\overline{\triangleright} x) & = & h_{(2)}\triangleright 
((S^{-1}(h_{(1)})\triangleright b) x).
\end{eqnarray*}

One can restrict this twisted action to an ideal $A= 1_A \; B$, where $1_A$ is a central idempotent in $B$. This generates a twisted partial action of $H$ on $A$ given by the partial measuring
\[
h\cdot a =1_A (h\triangleright a) 
\]
and the partial 2-cocycle 
\[
\omega (h , k)  =  (h_{(1)} \cdot 1_A ) u(h_{(2)} ,k_{(1)})  (h_{(3)}k_{(2)} \cdot 1_A ) .
\]
The proof that this indeed defines a twisted partial action is basically the same as given in \cite{ABDP}.
\end{ex}

We recall the definition of partial crossed product introduced in \cite{ABDP}. Given a twisted partial action $(H,A,\cdot, \omega)$, there is an associative product defined on the vector space $A \otimes H :$
\begin{equation*}(a \otimes h)(b \otimes k)
= a (h_{(1)} \cdot b) \omega (h_{(2)}, k_{(1)}) \otimes h_{(3)} k_{(2)}.
\end{equation*}
The algebra $A \otimes H $ is not  unital in general. Nevertheless, it contains a subalgebra $A \#_\omega H = (A \otimes H)(\um \otimes 1_H)$ with unit element $(\um \otimes 1_H),$ called the  partial crossed product. As a subspace of $A \otimes H$, it is generated by the elements
\[
a \# h = a(h_{(1)} \cdot \um ) \otimes h_{(2)}.
\]

Given any global 2-cocycle $v: H \otimes H \rightarrow \kappa$, where $\kappa$ has the  global $H$-action, $h\cdot 1=\epsilon (h)$,  one usually transfers the algebra structure of the crossed product $\kappa \#_v H$ to $H$ via the canonical isomorphism $\kappa \otimes H \rightarrow H$. This algebra is denoted by  ${}_v H$ and the product is given by 
\[
h \bu _v k = v(h_{(1)},k_{(1)})h_{(2)}k_{(2)}.
\]

If  $(H,\kappa, \cdot, \omega)$ is a twisted partial action then the canonical isomorphism $\kappa \otimes H \simeq H$ defines on $H$ the new product
\[
h \bu_\omega k = (h_{(1)} \cdot 1)\omega (h_{(2)},k_{(1)})h_{(3)}k_{(2)}.
\]
We will denote this algebra by ${}_\omega H$. 
The partial crossed product, which will be denoted by ${}_\omega \underline{H}$, can be identified with the subspace of $H$ generated by the elements $(h_{(1)} \cdot 1)h_{(2)}$, and the product of generators is given by 
\[
((h_{(1)} \cdot 1)h_{(2)} ) \bu _\omega ((k_{(1)} \cdot 1)k_{(2)} ) = (h_{(1)} \cdot 1)\omega (h_{(2)},k_{(1)}) h_{(3)}k_{(2)}.
\]
Later on, we will describe the algebra ${}_\omega \underline{H}$ when $H = \kappa G$ and $H = (\kappa G)^*$.

Recall that  a twisted partial action of a group $G$ on a unital $\kappa$-algebra $A$  is defined in  \cite{DES} as a triple
\[
\left(\{D_g\}_{g\in G}, \{\alpha_g\}_{g\in G},
\{w_{g,h}\}_{(g,h)\in G\times G}\right),
\]
where for each $g, h \in G$, $D_g$ is an ideal of $A,$  $\alpha_g  :  D_{g^{-1}} \to D_g $ is a $\kappa $-algebra isomorphism, $w_{g,h}$ is an invertible multiplier of $D_g D_{gh},$ and   some properties are satisfied.
 If each $D_g$ is  generated by a central idempotent $1_g,$ then the twisted partial actions is called unital, and, as it was pointed out in \cite{ABDP}, this matches our concept  of a partial action of the group algebra $\kappa G$ over $A .$ In this case, $1_g =g\cdot 1_A,$ and  it is reasonable to have under consideration partial actions of a Hopf algebra $H$ over some unital algebra $A$ such that the map ${\bf e}\in \mbox{Hom}_k (H,A)$, given by ${\bf e}(h)=(h\cdot \um )$,  is central with respect to the convolution product. These partial actions are, in some sense, more akin to partial group actions.

Furthermore, in the case of a unital twisted partial group action the  $\omega_{g,h}$'s are invertible elements in $D_g D_{gh}$, for all $g,h\in G$. In particular, if the group action is global, then every element $\omega_{g,h}$ is  invertible  in $A$, which is automatically translated into the Hopf algebra setting by saying that the cocycle $\omega \in \mbox{Hom}_k (H\otimes H ,A)$ is convolution invertible. In the partial case, we have to consider more suitable conditions to replace the convolution invertibility for the cocycle,  and the idea is to define a unital ideal in the convolution algebra $\mbox{Hom}_{\kappa} (H\otimes H , A),$ in which the partial 2-cocycle lives and has an inverse \cite{ABDP}.

Let $A = (A, \cdot, \omega)$ be a twisted partial $H$-module algebra. It is easy to see that the linear maps $f_1, f_2 :H\otimes H \rightarrow A$, defined by $f_1(h,k) = (h \cdot \um)\epsilon(k) $  and $f_2(h,k) = (hk \cdot \um) $, are both (convolution) idempotents in $Hom(H \ot H,A)$. We also have that ${\bf e}$ is an idempotent in $\Hom(H,A)$ (and $f_1(h,k) = {\bf e}(h) \epsilon(k)$). Notice that if $f_1$ and $f_2$ are central in the convolution algebra $\mbox{Hom}_{\kappa} (H\otimes H, A),$ then the partial 2-cocycle $\omega$ lies in the unital ideal $\langle f_1 \ast f_2 \rangle \trianglelefteq \mbox{Hom}_{\kappa} (H\otimes H, A),$ thanks to (TPA3) and (\ref{Prop1Israel}).

\begin{defi} \label{defi:symmetric} \cite{ABDP} Let $H$ be a $\kappa$-bialgebra (or Hopf algebra) and $A$ be a unital $\kappa$-algebra. A symmetric twisted partial action of $H$ on $A$ is given by the data $(H,A, \cdot ,\omega , \omega' )$, such that
\begin{enumerate}
\item[(STPA1)] $(H, A, \cdot , \omega)$ is a twisted partial action of $H$ on $A$.
\item[(STPA2)] The above defined maps $f_1$ and $f_2$ are central in the convolution algebra $\mbox{Hom}_{\kappa} (H\otimes H, A)$.
\item[(STPA3)] There exists a linear map $\omega' \in \langle f_1 \ast f_2 \rangle$ such that $\omega \ast \omega' =\omega' \ast \omega =f_1 \ast f_2$.
\end{enumerate}
The algebra $A$ is called a symmetric twisted partial $H$-module algebra.
\end{defi}

Observe that the unit element  of the ideal $\langle f_1 \ast f_2 \rangle$ is the central idempotent $f_1 \ast f_2$ itself, which thanks to (PM3) reads $f_1 \ast f_2 (h,k)=h\cdot (k\cdot \um )$. Note also that the centrality of $f_1$ in $\mbox{Hom}_{\kappa} (H\otimes H, A)$, automatically implies the centrality of ${\bf e}\in \mbox{Hom}_k (H,A)$. With the aid of the inverse $\omega'$ on can rewrite the condition (TPA2) in  Definition \ref{defi:twisted} as
\begin{equation} \label{twisting:omega-inverse}
h \cdot (k \cdot a) =  \omega(h_{(1)} , k_{(1)})(h_{(2)}k_{(2)} \cdot a)\omega'(h_{(3)} , k_{(3)}). 
\end{equation}

\begin{ex} Any (global) twisted action of a Hopf algebra $H$ on a unital algebra $A$ is automatically symmetric. Indeed, as $h\cdot \um  =\epsilon (h) \um $, we have $f_1 (h,k)=f_2 (h,k)=\epsilon (h) \epsilon (k) \um$. Then both $f_1$ and  $f_2$ are equal to the unit of the convolution algebra $\mbox{Hom}_k (H\otimes H, A)$, which can be absorbed on the left or on the right. More generally, given any (possibly non-unital) twisted $H$-module algebra $B$, with (global) twisted action given by the measuring map $\triangleright :H\otimes B \rightarrow B$ and an invertible and normalized 2-cocyle $u:H\otimes H \rightarrow M(B),$ one can define a twisted partial action of $H$ on any unital ideal $A\trianglelefteq B$, generated by a central idempotent $\um \in B,$ according to Example~\ref{restriction}. Moreover, if $f_1$ and $f_2$ are central in $\mbox{Hom}_{\kappa} (H\otimes H, \kappa)$, then this twisted partial action is symmetric, with the inverse (partial) cocycle given by
\[
\omega' (h , k)  =   (h_{(1)}k_{(1)} \cdot \um ) u^{-1}(h_{(2)} , k_{(2)}) (h_{(3)} \cdot \um ). 
\]
\end{ex}

\begin{remark}\label{rem:co-commutativeCase}  If $H$ is co-commutative and $A$ is commutative, then  for every twisted partial action of $H$ on $A$ the functions $f_1$ and $f_2$ are clearly central, because the convolution algebra $\mbox{Hom}_k (H\otimes H, A)$ is commutative.  Moreover, if  the twisted partial action is symmetric (i.e. satisfies (STPA3)), then  the twisting disappears from   equality (\ref{twisting:omega-inverse}), which now  simply means that $h\cdot (k\cdot a)=(h_{(1)}\cdot \um )(h_{(2)}k \cdot a)$, resulting in  a partial action of $H$ on $A$. Thus in this case a symmetric twisted partial action disassembles into  a pair consisting of a partial action of $H$ on $A$ and a partial $2$-cocycle $\omega $ with  (STPA3).  A partial $2$-cocycle means   now a $\kappa $-linear map   $\omega:H\otimes H\to A$ satisfying  conditions (TPA3)-(TPA5) from Definition~\ref{defi:twisted}.
\end{remark}

\begin{ex} \label{exemplo:grupo2}

Let us consider the specific case of twisted partial actions of a group algebra $\kappa G$ over a field $\kappa$. Then   every twisted partial action of $\kappa G$ on $\kappa$ is automatically symmetric. Indeed, as we know from the Example \ref{exemplo:grupo}, the measuring maps are classified by subgroups $L\leq G$ such that $\lambda_g =1$ if $g\in L$ and $\lambda_g =0$ otherwise. One can classify every twisted partial action of $\kappa G$ on $\kappa$, determining the possible values for the partial $2$-cocycle $\omega.$ Equality  (\ref{Prop1Israel}) gives us, 
\[
\omega (g,h) =\lambda_g \lambda_h \omega (g,h), \quad \forall g,h\in G .
\]
Then, automatically $\omega (g,h)=0$ if one of its two entries does not belong to the subgroup $L$. The cocycle condition (TPA5) in Definition \ref{defi:twisted} reads
\[
\lambda_g \omega (h,k) \omega (g,hk)= \omega (g,h)\omega (gh,k) ,
\]
which, taking $g,h,k\in L$ transfers into   the classical 2-cocycle condition for the group cohomology of the group $L$ with values in the trivial $L$-module $\kappa ^{\ast}:$ 

\[
 \omega (h,k) \omega (g,hk)= \omega (g,h)\omega (gh,k).
\] Conversely, given a  classical $2$-cocycle $\omega : L \times L \to    \kappa ^{\ast}, $  one may extend $\omega $ to  $\kappa G \otimes \kappa G,$  by setting  $\omega   (g,h) = 0$ if $g$ or $h$   does not belong to  $L,$ and obtaining thus a partial $2$-cocycle. Hence the partial $2$-cocycles of $\kappa G$ related to the given partial action of $\kappa G$ on $\kappa $ can be identified with the classical $2$-cocycles of the subgroup $L$ with values in  $\kappa ^{\ast}.$  Obviously, setting ${\omega}'(g,h)={\omega}\m(g,h)$ if and only if both $g$ and $h$ lie in $L, $ we see that 
the twisted partial action $(\kappa G, \kappa , \cdot , \omega )$ is necessarily   symmetric.

In order to describe the  crossed product in this case, let $v$ be a 2-cocycle of $L$ and let $\omega$ be its extension to $\kappa G$ as above. The vector space ${}_\omega \underline{\kappa G}$ is generated by the elements $\lambda_g g$, where $\lambda_g = 1$ if $g \in L$ and zero otherwise, and hence ${}_\omega \underline{\kappa G} = \kappa L$ as a vector space. It also follows that
${}_\omega \underline{\kappa G} = {}_v \kappa L$ as an algebra. Observe that  ${}_v \kappa L$ is the twisted group ring of $L$ over $\kappa $ corresponding the $2$-cocycle $v.$
\end{ex}

\begin{ex} The only symmetric twisted partial actions of the Sweedler Hopf algebra $H=H_4$ on a $\kappa $-algebra $A$   are the global ones. Indeed, for any linear function $\phi :H_4 \otimes H_4 \rightarrow A$, we have
\[
f_1 \ast \phi (x,g) =f_1 (x,g)\phi (1_H,g) +f_1 (g,g) \phi (x,g)= (x \cdot \um)\phi (1_H,g)+ (g\cdot \um) \phi (x,g),
\]
while, on the other hand,
\[
\phi \ast f_1 (x,g)=\phi (x,g) f_1 (1_H,g)+\phi (g,g) f_1 (x,g)=\phi (x,g)  +  \phi (g,g) (x\cdot \um ).
\] Taking $\phi $ with $(g,g) \mapsto \um,$  $(x,g) \mapsto 0,$ and $(1_H,g) \mapsto 0,$  we obtain that $x \cdot \um =0= \epsilon (x) \um ,$  and for  $\phi $ with $(g,g) \mapsto 0,$  $(x,g) \mapsto \um,$ and $(1_H,g) \mapsto 0,$ we get $g \cdot \um = \um  =  \epsilon (g) \um .$ Then by (PM3), 
$$ 0= g\cdot (x \cdot \um) = (g \cdot \um ) ( gx \cdot \um ) = gx \cdot \um,$$ so that $gx \cdot \um = \epsilon (gx) \um, $ and by linearity, $h \cdot \um = \epsilon (h) \um$ for all $h \in H_4.$   Consequently, $f_1$ is central in ${Hom}_{\kappa} (H\otimes H, A)$
exactly when the twisted partial action is global. 

\end{ex}

\begin{ex}  The above example  shows, in particular,  that there is no symmetric twisted partial action of the Sweedler Hopf algebra over the base field $\kappa $ which is not global. Nevertheless, we shall  construct a partial cocycle defining a (non-global) twisted partial action of $H=H_4$ on $\kappa .$

As we have seen in  Example \ref{exemplo:sweedler}, the truly partial measuring maps of $H_4$ on $\kappa$ are given by functionals $\lambda :H_4 \rightarrow \kappa$ such that $\lambda_g =0$ (recall that $\lambda_x$ may have any value in $\kappa $ and $\lambda_x =\lambda_{gx} = - \lambda_{xg}$). Take such a  partial measuring, then  the normalization condition (TPA4) for the partial cocycle gives us
\[
\omega (1_H, h)=\omega (h,1_H)=\lambda _h, \quad \forall h\in H_4 ,
\]
which means that $\omega (1_H,1_H)=1$, $\omega (1_H,g)=\omega (g,1_H)=0$, $\omega (x,1_H)=\omega (1_H,x)=-\omega (xg ,1_H)=-\omega (1_H,xg)=\lambda_x$. The axioms (TPA2) and (TPA3), combined, allows us to write the identities,
\begin{equation}\label{TPA2,3(a)}
\omega (h,k) =\lambda_{h_{(1)}} \omega (h_{(2)}, k) 
\end{equation} and
\begin{equation}\label{TPA2,3(b)}
\omega (h,k)=\lambda_{h_{(1)}} \lambda_{k_{(1)}} \omega (h_{(2)} ,k_{(2)}).
\end{equation} Then (\ref{TPA2,3(a)}) gives   
\begin{align*}
& \omega (g,h)  =  \lambda_g \omega (g,h)=0, \quad \forall h\in H_4 , \\
&\omega (x,x)  =  \lambda_x \omega (1_H,x)+\lambda_g \omega (x,x)=\lambda_x^2 ,\\
& \omega (x,xg)   =  \lambda_x \omega (1_H,xg) +\lambda_g \omega (x,xg)=-\lambda_x^2 ,\\
\end{align*} and (\ref{TPA2,3(b)}) results in 

\begin{align*}
& \omega (h,g)  =  \lambda_{h_{(1)}} \lambda_g \omega (h_{(2)} ,g) =0 , \quad \forall h\in H_4 ,\\
& \omega (xg,x) =  \lambda_{xg} \lambda_x \omega (g,1_H)+\lambda_{xg} \lambda_g \omega (g,x)+\lambda_1 \lambda_x \omega (xg ,1_H) +\lambda_1 \lambda_g \omega (xg,x)=-\lambda_x^2.
\end{align*}
Moreover (\ref{TPA2,3(a)})  and (\ref{TPA2,3(b)}) do not give any  restriction for the value of $\omega (xg,xg).$ Then taking arbitrary $\omega (xg,xg)\in \kappa $ it is readily seen that $\omega : H_4 \otimes H_4 \to \kappa $ also satisfies (TPA3), and consequently (TPA2).    A  list of verifications certify that the cocycle identity (TPA5) holds for $\omega .$        
\end{ex}

\begin{ex} Given a finite group $G$ and a field $\kappa $ with ${\rm char}\; \kappa \nmid |G|,$ in order to classify the symmetric twisted partial actions of $(\kappa G)^*$ on  $\kappa$, one needs first to analyse  the  centrality  of $f_1$ and $f_2.$ Consider a partial measuring map $\lambda :(\kappa G)^* \rightarrow \kappa$ determined by the subgroup $L$, such that $\lambda (p_g )=\frac{1}{|L|}$ if $g\in L$ and $\lambda (p_g )=0$ otherwise. Take any linear function $\phi : (\kappa G)^* \otimes (\kappa G)^* \rightarrow \kappa$, and $g,h\in G$, then
\[
f_1 \ast \phi (p_g, p_h )= \sum_{r,s\in G} f_1 (p_r ,p_s )\phi (p_{r^{-1}g}, p_{s^{-1}h}) =\sum_{r\in L} \lambda ({p_r}) \phi (p_{r^{-1}g}, p_h ) ,
\]
and
\[
\phi \ast f_1 (p_g, p_h )= \sum_{r,s\in G} \phi (p_{gr^{-1}}, p_{hs^{-1}}) f_1 (p_r, p_s )=\sum_{r\in L} \lambda ({p_r}) \phi (p_{gr^{-1}} , p_h ).
\]
Then the equality $f_1 \ast \phi =\phi \ast f_1$ is true for every $\phi$, if, and only if the subgroup $L$ is contained in the center of $G.$ It is readily seen that in this case $f_2$ is also central.

The second step is to define a normalized partial 2-cocycle $\omega :(\kappa G)^* \otimes (\kappa G)^* \rightarrow \kappa$ which is invertible in the unital ideal generated by $f_1 \ast f_2$. If $\omega \ast f_1 \ast f_2 =\omega$, then it is easy to see that $\omega =\omega \ast f_1 =\omega \ast f_2$. Take $g,h\in G$, then
\begin{eqnarray*}
\omega (p_g , p_h ) & = & f_1 \ast \omega (p_g , p_h ) =  \frac{1}{|L|} \sum_{r\in L} \omega (p_{r^{-1}g}, p_h) .
\end{eqnarray*}
Therefore $\omega (p_g ,p_h )=\omega (p_{kg}, p_h )$ for any $k\in L$. For the identity $\omega =\omega \ast f_2$, we have
\begin{eqnarray*}
\omega (p_g , p_h ) & = & f_2 \ast \omega (p_g , p_h ) =  \frac{1}{|L|} \sum_{r\in L} \omega (p_{r^{-1}g}, p_{r^{-1}h}),
\end{eqnarray*}
which leads to the conclusion that $\omega (p_g ,p_h )=\omega (p_{kg}, p_{kh} )$ for any $k\in L$. Joining these two informations, we obtain 
\[
\omega (p_g ,p_h )=\omega (p_{kg} , p_{lh}), \quad \forall k,l \in L ,
\]
and considering the fact that $L$ is a subgroup of the center of $G$, then also
\[
\omega (p_g ,p_h )=\omega (p_{gk} , p_{hl}), \quad \forall k,l \in L .
\]  
Let us say that a $\kappa $-linear function  $\omega : (\kappa  G)^* \otimes (\kappa  G)^*\to \kappa $ is $L$-invariant if $\omega $ satisfies the above two equalities. Notice that  the above computations show that a  $\kappa $-linear function  $\omega : (\kappa  G)^* \otimes (\kappa  G)^*\to \kappa $ is $L$-invariant exactly when $\omega $ is contained in the ideal $\langle f_1 \ast f_2 \rangle $ of $ \mbox{Hom}_{\kappa} ((\kappa  G)^* \otimes (\kappa  G)^*, A)$.

Taking into account this  $L$-invariance of the partial 2-cocycle, the first member of the cocycle identity 
\[
(m_{(1)}\cdot \omega (n_{(1)} , l_{(1)}))\omega (m_{(2)}, n_{(2)}l_{(2)})=\omega(m_{(1)} ,n_{(1)})\omega (m_{(2)}n_{(2)} ,l) ,
\]
for $m=p_g$, $n=p_h$ and $l=p_k$, reads 
\begin{eqnarray*}
(m_{(1)}\cdot \omega (n_{(1)} , l_{(1)}))\omega (m_{(2)}, n_{(2)}l_{(2)}) & = & \sum_{r,s,t\in G} \lambda (p_r ) \omega (p_{hs^{-1}}, p_{kt^{-1}}) \omega (p_{r^{-1}g}, p_s p_t ) \\
& = & \sum_{s\in G} \sum_{r\in L} \frac{1}{|L|} \omega (p_{hs^{-1}}, p_{ks^{-1}}) \omega (p_{r^{-1}g}, p_s )\\
& = & \sum_{s\in G}  \omega (p_{hs^{-1}}, p_{ks^{-1}}) \left( \sum_{r\in L}\frac{1}{|L|}  \omega (p_{g}, p_s ) \right) \\
& = & \sum_{s\in G}  \omega (p_{hs^{-1}}, p_{ks^{-1}}) \omega (p_{g}, p_s ) .
\end{eqnarray*}
The second member, in its turn, is
\begin{eqnarray*}
\omega(m_{(1)} ,n_{(1)})\omega (m_{(2)}n_{(2)} ,l) & = & \sum_{s,t \in G} \omega (p_{gs^{-1}}, p_{ht^{-1}}) \omega (p_s p_t, p_k )\\
& = & \sum_{s\in G} \omega (p_{gs^{-1}}, p_{hs^{-1}}) \omega (p_s , p_k ) .
\end{eqnarray*}
Therefore, the cocycle identity (TPA5)  transforms into
\begin{equation}\label{ParGlobCocycle}
\sum_{s\in G}  \omega (p_{hs^{-1}}, p_{ks^{-1}}) \omega (p_{g}, p_s ) =\sum_{s\in G} \omega (p_{gs^{-1}}, p_{hs^{-1}}) \omega (p_s , p_k ),
\end{equation} for all $g,h,k \in G.$    {It is readily seen that (\ref{ParGlobCocycle}) means that $\omega $ is a global $2$-cocycle of $(\kappa G)^*$ with respect to the trivial action on $\kappa .$  Let $G\setminus  L$ be a full set of representatives of left ($=$ right) cosets of $G$ by $L.$ Then an $L$-invariant  function $\omega \in   \mbox{Hom}_{\kappa} ((\kappa  G)^* \otimes (\kappa  G)^*, A)$ satisfies (\ref{ParGlobCocycle}) if and only if 
\begin{equation}\label{ParGlobCocycle2}
\sum_{s\in G\setminus L}  \omega (p_{hs^{-1}}, p_{ks^{-1}}) \omega (p_{g}, p_s ) =\sum_{s\in G\setminus L} \omega (p_{gs^{-1}}, p_{hs^{-1}}) \omega (p_s , p_k ),
\end{equation} for all $g,h,k \in G\setminus  L.$ 

Now, let $\omega : (\kappa  G)^* \otimes (\kappa  G)^*\to \kappa $ be a partial $2$-cocycle, so that our partial measuring, determines by the central subgroup $L\leq G,$  is a symmetric twisted partial action. Then  taking $v(p_{gL}, p_{hL}) = |L|^2 \omega (p_g, p_h), $ we have,  thanks to the $L$-invariance of $\omega,$ a well-defined $\kappa$-linear map $v: (\kappa  \, G/L )^* \otimes (\kappa  \, G/L )^*\to \kappa  ,$ which in view of (\ref{ParGlobCocycle2})   satisfies the global $2$-cocyle equality with respect to the trivial action on $\kappa .$ The normalization condition (TPA4) of $\omega $ directly implies that $v$ is normalized. Furthermore, since our twisted partial action is symmetric, there exists $\omega ' : (\kappa  G)^* \otimes (\kappa  G)^*\to \kappa $ satisfying (STPA3). Then taking   $u(p_{gL}, p_{hL}) = |L|^2 \omega ' (p_g, p_h), $ we compute that 
\begin{align*} & (v\ast u ) (p_{gL}, p_{hL}) = \sum _{s,t\in G\setminus L} v (p_{g s\m L}, p_{h t\m L})  u  (p_{sL}, p_{tL}) =
|L|^4\sum _{s,t\in G\setminus L} \omega (p_{g s\m }, p_{h t\m })  \omega '  (p_{s}, p_{t})= \\
&\frac{|L|^4}{|L|^2}\sum _{s,t\in G\setminus L} \sum _{l, m\in L }\omega (p_{g s\m l\m }, p_{h t\m m\m})  \omega '  (p_{sl}, p_{tm})= 
|L|^2 \sum _{x, y\in G} \omega (p_{g x\m }, p_{h y\m})  \omega '  (p_{x}, p_{y})=\\
& |L|^2 (\omega \ast \omega ')(p_g, p_h)=  |L|^2 \lambda (p_g) \lambda ( p_h)= \epsilon (p_{gL}) \epsilon (p_{hL}), 
\end{align*} showing that $u$ is the convolution inverse of $v.$ Consequently, $v$ is a normalized global convolution invertible  $2$-cocycle of $(\kappa \, G/L )^*.$

Conversely, given     a normalized convolution invertible global 2-cocycle  $v$  of $(\kappa \, G/L)^*$ with respect to the trivial action of  
$(\kappa \, G/L)^*$ on $\kappa ,$  define the $\kappa $-linear functions $ \omega  , \omega ' : (\kappa  G)^* \otimes (\kappa  G)^*\to \kappa $  by 
$$\omega (p_g , p_h)=\frac{1}{|L|^2} v(p_{gL}, p_{hL})\;\; \text{ and }\;\;  \omega '(p_g , p_h)=\frac{1}{|L|^2} v\m (p_{gL}, p_{hL}).$$ Then  the above considerations imply   that 
we obtain a symmetric twisted partial action  of $(\kappa  G)^*$ on $ \kappa. $

Thus we have one-to-one correspondences between the following three sets:
 \begin{itemize}
\item The partial $2$-cocycles  $\omega : (\kappa  G)^* \otimes (\kappa  G)^*\to \kappa $  with respect to the partial action  of  $(\kappa  G)^*$ on $\kappa ,$ determined by the central subgroup $L \subseteq G,$ which are invertible in the ideal $\langle f_1 \ast f_2 \rangle .$

\item The $L$-invariant global  $2$-cocycles  $\omega : (\kappa  G)^* \otimes (\kappa  G)^*\to \kappa $ with respect to the trivial action  of  $(\kappa  G)^*$ on $\kappa ,$ which satisfy the equalities (TPA4) and (STPA3).

\item The  normalized convolution invertible  global $2$-cocycles  $ v : (\kappa  \, G/L)^* \otimes (\kappa \,  G/L)^*\to \kappa $ with respect to the trivial action  of  $(\kappa  \, G/ L)^*$ on $\kappa .$

 \end{itemize}

Let now $\omega $ be a partial $2$-cocycle such that we have a symmetric twisted partial action (i.e. $\omega $ is as in the first item above), 
and let $v  : (\kappa  \, G/L)^* \otimes (\kappa \,  G/L)^*\to \kappa $ be the corresponding global $2$-cocyle. Then  the map 
\[ \phi : {}_v(\kappa \, G/L)^* \rightarrow   
{}_\omega\underline{(\kappa G)^*}, \;\;\;\;\; 
p_{gL}  \mapsto  \sum_{x \in L} p_{gx}
\]
is an isomorphism of algebras. } In fact, as a vector space, ${}_\omega\underline{(\kappa G)^*}$ is generated by the elements
\[
 \sum _{x\in G} \lambda(p_{x^{-1}}) p_{gx} = \dfrac{1}{|L|}
\sum_{x \in L} p_{gx}.
\]
If $G \setminus L $ is a transversal for $L$ in $G$ then the elements 
\[
\sum_{x \in L} p_{tx}, \;\;\; (t\in G \setminus L),
\]
form a basis of ${}_\omega\underline{(\kappa G)^*}$, which is the image under $\phi$ of the basis $ \{p_{tL}: \; t\in G \setminus L  \}$ of 
${}_v(\kappa \,G/L)^*$, and therefore $\phi$ is a linear isomorphism.

Using the $L$-invariance of $\omega $ we see that the algebra structure on  ${}_\omega(\kappa G)^*$  is defined on basis elements by 
\[
p_g \bu _\omega p_h =  \sum _{r,s,t\in G} \lambda (p_r) \omega (p_{r\m g s\m}, p_{ht\m}) p_s p_t = \sum_{t \in G} \omega(p_{gt^{-1}},p_{ht^{-1}}) p_t.
\]  
From this expression, the $L$-invariance and the definition of $\omega$, it follows that 
\begin{align*}
 \phi (p_{t L})  \bu _\omega \phi (p_{t' L})
 & =    \sum_{s \in G } \sum_{l,m \in L} \omega(p_{t l s^{-1}}, p_{t' ms^{-1} }) p_{s}\\
 & =  |L| ^2  \sum_{s \in G} \omega(p_{t s^{-1} },p_{t' s^{-1} }) p_{s}\\
  & =  |L| ^2  \sum_{s \in G\setminus L}  \omega(p_{t s^{-1} },p_{t' s^{-1} })  \sum _{x \in L}  p_{sx}\\
& = \phi ( |L| ^2  \sum_{s \in G\setminus L}  \omega(p_{t s^{-1} },p_{t' s^{-1} })   p_{s L} )\\
& = \phi (   \sum_{s \in G\setminus L}  v (p_{t s^{-1} L},p_{t' s^{-1}L })   p_{s L} )\\
& = \phi (p_{t L}  \bu _v p_{t' L}), 
\end{align*} with $t, t'\in G \setminus L .$ In addition, it is trivially seen that $\phi $ is unital. 
\end{ex}

\begin{ex} \label{exemplo:klein4} Consider the previous example for the specific case when the group $G$ is the Klein four-group $K_4 =\langle a,b \; | \; a^2 =b^2 =e \rangle .$ Take the subgroup $L=\langle a \rangle$, then the values of the partial measuring map $\lambda$ are $\lambda (p_e ) =\lambda (p_a )=\frac{1}{2}$ and $\lambda (p_b )=\lambda (p_{ab} )=0$. As the group $K_4$ is abelian, the subgroup $L$ is automatically in the center, then we have no obstruction to the centrality of $f_1$ and $f_2$ (moreover, since $K_4$ is abelian,  $(\kappa K_4)^{\ast} $ is co-commutative and, consequently the convolution algebra $\mbox{Hom}_k ((\kappa K_4)^{\ast}\otimes (\kappa K_4)^{\ast}, \kappa)$ is commutative). The  $L$-invariance of the partial $2$-cocycle reads
\begin{eqnarray*}
& \, & \omega (p_e ,p_e )=\omega (p_a , p_e ) =\omega (p_e , p_a ) =\omega (p_a , p_a ) , \\
& \, & \omega (p_e ,p_b )=\omega (p_a , p_b ) =\omega (p_e , p_{ab} ) =\omega (p_a , p_{ab} ) , \\
& \, & \omega (p_b ,p_e )=\omega (p_{ab} , p_e ) =\omega (p_b , p_a ) =\omega (p_{ab} , p_a ) , \\
& \, & \omega (p_b ,p_b )=\omega (p_{ab} , p_b ) =\omega (p_b , p_{ab} ) =\omega (p_{ab} , p_{ab} ) .
\end{eqnarray*}
The normalization condition, $\omega (1_H ,h)=\omega (h, 1_H )=h\cdot \um$, gives us three conditions
\begin{eqnarray*}
\omega (p_e , p_e ) +\omega (p_e ,p_b ) & = & \frac{1}{4} , \\
\omega (p_e , p_e ) +\omega (p_b ,p_e ) & = & \frac{1}{4} , \\
\omega (p_e , p_b ) +\omega (p_b ,p_b ) & = & 0 .
\end{eqnarray*}
Which leads to $\omega (p_b ,p_e )=\omega (p_e, p_b ) =-\omega (p_b , p_b )$, and $\omega (p_e ,p_b )=\frac{1}{4}-\omega (p_e , p_e )$. Therefore, we end up with only one independent value of the partial 2-cocycle, $x=\omega (p_e , p_e)$. One directly checks that with these data, the $2$-cocycle condition is satisfied for arbitrary value of  $x$. Similarly,   $\omega' $ is also $L$-invariant and normalized, and can be determined in a similar form in terms of  $y=\omega '(p_e , p_e).$ Furthermore, a direct computation shows that  the condition $\omega \ast \omega' =f_1 \ast f_2  $ is equivalent to  the equality    
\begin{equation}\label{points}
32xy-6(x+y) +1=0.
\end{equation}
Therefore the symmetric twisted partial actions of $(\kappa K_4)^{\ast}$ on $\kappa $ are parametrized by the  points $(x,y)\in \kappa^2 $ satisfying (\ref{points}). 
\end{ex}

\section{Globalization}

As it was  already seen, taking a twisted action of a Hopf algebra $H$ over an algebra $B$ and then 
restricting it to a unital ideal $A$, one obtains a twisted partial Hopf action of $H$ on $A$. Now, 
given a twisted partial action of a Hopf algebra $H$ on a unital algebra $A$, we can ask what conditions 
must be satisfied on order to view it as a restriction of a global action. The globalization problem was 
solved in several contexts. In particular,  for partial group actions on $C^*$-algebras  
the answer was given by F. Abadie  (see \cite{Abadie} or \cite{AbadieTwo}). For the case of partial group 
actions on unital rings a criterion for the existence of a globalization was given  in  \cite{DE}, whereas  for 
partial Hopf actions the globalization always exists, as it was shown in \cite{AB}. As to the twisted partial group actions,
the globalization is quite more involved as it can be seen  in  \cite{DES2}.  
We recall briefly the meaning of a globalization of a twisted partial group action.  
We denote by ${\cal U}(B)$ the group of invertible elements of an algebra $B.$ Given a twisted partial action of $G$ on an algebra $A$,
$\{ \{ D_g \}_{g\in G} , \{ \alpha_g:D_{g^{-1}}\rightarrow D_g \}_{g\in G}, 
\{ w_{g,h} \in {\cal U} (D_g D_{gh})  \}_{g,h\in G} \}$, a globalization for this twisted partial action is a quadruple $\{B, \beta :G\rightarrow \mbox{Aut}(B), \varphi , u:G\times G \rightarrow {\cal U}(M(B)) \}$,  where $M(B)$ is the multiplier algebra of $B$,
 
such that
\begin{enumerate}
\item $(B,\beta ,u)$ is a twisted action of $G$ on $B$ with cocycle $u$.
\item $\varphi :A\rightarrow B$ is a monomorphism of algebras and $\varphi (A)\trianglelefteq B$.
\item $\varphi (D_g )=\varphi (A)\cap \beta_g (\varphi (A))$.
\item $B=\sum_{g\in G} \beta_g (\varphi (A))$.
\item $\varphi$ intertwines the twisted partial action of $G$ on $A$ with the induced twisted partial action on $\varphi (A)$ obtained by the restriction of the twisted action $\beta$ on $B$.
\end{enumerate}
In this section we will follow some ideas present in \cite{DES2} to construct a globalization  for a twisted partial action of a Hopf algebra $H$ on a unital algebra $A$.

\begin{defi}
Let $(A,(w,w'))$, $(A',(v,v'))$ be two symmetric twisted partial $H$-module algebras. A map $\varphi: A \rightarrow A'$ is an isomorphism of symmetric twisted partial $H$-module algebras if 
\begin{enumerate}[\rm(i)]
\item $\varphi$ is an algebra isomorphism; 
\item  $\varphi(h \cdot a) = h \cdot \varphi(a)$;
\item  $\varphi (w (h,k) ) = v(h,k)$, $\varphi(w'(h,k) ) = v' (h,k),$ 
\end{enumerate} for all $h,k \in H $ and $a\in A.$ 
\end{defi}

\begin{defi} Let $A$ be a symmetric twisted partial $H$-module algebra with the pair $(w,w')$. 
A globalization of $A$ is a pair $(B, \varphi)$, where $B$ is a (possibly non-unital) twisted $H$-module algebra  with invertible cocycle 
$u: H \otimes H \to M(B)$ and $\varphi: A \rightarrow B$ is an algebra monomorphism  such that 
\begin{enumerate}[\rm(i)]
\item $\varphi(A)$ is an ideal in $B$; 
\item $\varphi: A \rightarrow \varphi(A)$ is an isomorphism of symmetric twisted partial $H$-module algebras, where $\varphi(A)$ has the structure induced by $B$. 
\end{enumerate}
If the algebra $B$ is unital then we call the globalization unital. 

\end{defi}

Our main result is as follows.

\begin{thm} \label{teorema:globalizacao} Let $A$ be a symmetric twisted partial $H$-module algebra with the pair $(w,w')$. 
\begin{enumerate}
 \item If this partial action has a globalization then 
 there is a convolution invertible $\kappa$-linear map $\wtil: H \otimes H \rightarrow A$ satisfying
\begin{equation}\label{wtil}
 (h_{(1)} \cdot \wtil (k_{(1)},l_{(1)}))\wtil (h_{(2)},k_{(2)}l_{(2)}) =  (h_{(1)} \cdot \um )\wtil (h_{(2)},k_{(1)}) \wtil (h_{(3)}k_{(2)},l)
\end{equation}
with $\wtil(1_H,h)=\wtil(h,1_H) = \epsilon(h) \um$,  such that 
 $\omega = (f_1 *f_2) *\wtil $ and $\omega' = (f_1 *f_2) *\wtil^{-1}$, i.e.,   
\begin{eqnarray}
\omega (h,k) & = &   (h_{(1)} \cdot \um)(h_{(2)}k_{(1)} \cdot \um ))\wtil (h_{(3)},k_{(2)}), 
\label{globcocycle1} \\
\omega' (h,k) & = &  (h_{(1)} \cdot \um)(h_{(2)}k_{(1)} \cdot \um ))\wtil^{-1} (h_{(3)},k_{(2)}). \label{globcocycle2}
\end{eqnarray}
\item If there is a convolution invertible linear map $\wtil$ as before, then $A$ admits a unital globalization. 
\end{enumerate}
\end{thm}

\begin{proof}

(1) Suppose that the pair $(B,\varphi )$ is a globalization of the twisted partial action of $H$ on $A.$ 
Then, there are a twisted action of $H$ on $B$ given by  $\rhd :H\otimes B\rightarrow B$ and a normalized, 
convolution invertible cocycle $u\in \mbox{Hom}(H\otimes H ,M(B))$, as well as    an algebra monomorphism $\varphi : A \rightarrow B $ such that $\varphi (A)$ is an ideal of $B$. The monomorphism $\varphi$ also intertwines the partial action of $H$ on $A$ with the induced partial action of $H$ on $\varphi (A)$, which is given by
\[
h\cdot \varphi (a) =\varphi (\um )(h\rhd \varphi (a)).
\]
and the induced partial cocycle is given by the expressions 
\begin{eqnarray*}
v(h,k) & = &   (h_{(1)} \cdot (k_{(1)} \cdot \varphi (\um )))u(h_{(2)},k_{(2)}) \nonumber \\
& = &  (h_{(1)} \cdot \varphi (\um )) u(h_{(2)},k_{(1)}) (h_{(3)}k_{(2)} \cdot \varphi ( \um )),   \\
v'(h,k) & = &  u^{-1}(h_{(1)},k_{(1)})(h_{(2)} \cdot (k_{(2)} \cdot \varphi (\um ))) \nonumber \\
& = &  (h_{(1)}k_{(1)} \cdot \varphi (\um ))u^{-1}(h_{(2)},k_{(2)})(h_{(3)} \cdot \varphi (\um )) .
\end{eqnarray*}
Therefore, we have $\varphi (h\cdot a)=h\cdot \varphi (a)$, $\varphi (\omega )=v$ and $\varphi (\omega' )=v'$. 

Define $\wtil \in \mbox{Hom}(H\otimes H ,A)$ such that
\[
\varphi (\wtil (h,k))=\varphi (\um )u(h,k).
\]
As the map $\varphi$ is a monomorphism, one can conclude that $\wtil$ is well defined. First, note that $\wtil$ is normalized, for
\[
\varphi (\wtil (h, 1_H)) =\varphi (\um )u(h,1_H)=\varphi (\um )\epsilon (h)= \varphi (\um \epsilon (h)) .
\]
By the injectivity of $\varphi$ we have $\wtil (h,1_H)=\epsilon (h)\um$. Similarly  $\wtil (1_H,h)=\epsilon (h)\varphi (\um )$. Also it is easy to see that $\wtil$ is convolution invertible, for  define $\wtil^{-1}$ by $\varphi (\wtil^{-1} )=\varphi (\um )u^{-1} (h,k)$, and then
\begin{eqnarray*}
\varphi ( \wtil ( h_{(1)} ,k_{(1)}) \wtil^{-1} (h_{(2)} ,k_{(2)})) &=& \varphi ( \wtil ( h_{(1)} ,k_{(1)}))\varphi ( \wtil^{-1} (h_{(2)} ,k_{(2)}))\\
&=& \varphi (\um )  u( h_{(1)} ,k_{(1)}) u^{-1} (h_{(2)} ,k_{(2)})\\
&=&  \varphi (\um ) \epsilon (h) \epsilon (k)\\
&=& \varphi (\um \epsilon (h) \epsilon (k)) .
\end{eqnarray*}
Again, by the injectivity of $\varphi$ we get 
\[
 \wtil ( h_{(1)} ,k_{(1)}) \wtil^{-1} (h_{(2)} ,k_{(2)}) =\um \epsilon (h) \epsilon (k) .
\]
The other inversion formula is obtained in a  similar way.

The modified cocycle expression (\ref{wtil}) is seen by
\begin{eqnarray*}
& & \varphi (  (h_{(1)} \cdot \wtil (k_{ (1) },l_{(1)}))\wtil (h_{(2)},k_{(2)}l_{(2)})) = \\
& =&  \varphi ( (h_{(1)} \cdot \wtil (k_{  (1)  },l_{(1)})))\varphi ( \wtil (h_{(2)},k_{(2)}l_{(2)})) \\
&=&   (h_{(1)} \cdot \varphi (\wtil (k_{   (1)    },l_{(1)})))\varphi ( \wtil (h_{(2)},k_{(2)}l_{(2)})) \\
&=&   (h_{(1)} \cdot (\varphi (\um ) u(k_{  (1)   },l_{(1)})))\varphi (\um )              u (h_{(2)},k_{(2)}l_{(2)}) \\
&=&   (h_{(1)} \cdot \varphi (\um )) (h_{(2)} \cdot u(k_{   (1)    },l_{(1)}))
\varphi (\um ) u (h_{(3)},k_{(2)}l_{(2)}) \\
&=&   (h_{(1)} \cdot \varphi (\um )) (h_{(2)} \overline{\rhd} u(k_{ (1)   },l_{(1)}))
u (h_{(3)},k_{(2)}l_{(2)}) \\
&\stackrel{(*)}{=} &   (h_{(1)} \cdot \varphi (\um )) u(h_{(2)},k_{(1)}) u(h_{(3)}k_{(2)} ,  l  ) \\
&=& (h_{(1)} \cdot \varphi (\um ))\varphi ( \wtil (h_{(2)},k_{(1)}) ) 
\varphi ( \wtil (h_{(3)}k_{(2)},  l    )) \\
&=& \varphi (  (h_{(1)} \cdot \um ) \wtil (h_{(2)},k_{(1)})  
\wtil (h_{(3)}k_{(2)}, l     )) , \\
\end{eqnarray*}
where in the equality (*) above we used the fact that the cocycle equation is an identity between multipliers.

Finally, for the expressions (\ref{globcocycle1}) and (\ref{globcocycle2}) we have
\begin{eqnarray*}
&& \varphi ( (h_{(1)} \cdot (k_{(1)} \cdot \um ))\wtil (h_{(2)} ,k_{(2)})) \\
& =&  (h_{(1)} \cdot (k_{(1)} \cdot \varphi (\um )))\varphi (\um ) u(h_{(2)} ,k_{(2)}) \\
& =&  (h_{(1)} \cdot (k_{(1)} \cdot \varphi (\um ))) u(h_{(2)} ,k_{(2)}) \\
&=& v(h,k) =\varphi (\omega (h,k)),
\end{eqnarray*}
which gives
\[
\omega (h,k)= (h_{(1)} \cdot (k_{(1)} \cdot \um ))\wtil (h_{(2)} ,k_{(2)}) .
\]
The other equalities are obtained similarly.

(2) 
Consider the map 
\begin{eqnarray*}
\varphi: A & \rightarrow & \Hom (H,A) \\
a & \mapsto & \varphi(a): h \mapsto h \cdot a .
\end{eqnarray*}
From (\ref{productpartial}) one can easily see that $\varphi$ is an algebra homomorphism. In addition, if $a\in \mbox{Ker}(\varphi )$, then 
\[
a=1_H \cdot a =\varphi (a)(1_H )=0,
\]
therefore, $\varphi$ is an algebra monomorphism.

We begin by proving that the algebra $\Hom(H,A)$ is a twisted $H$-module algebra, and then we show that there is a subalgebra $B$ of $\Hom(H,A)$, which inherits the structure of the twisted $H$-module and contains $\varphi(A)$ as an ideal.

Let us prove that $\Hom (H,A)$ is a twisted $H$-module algebra. Firstly,  
$H$ measures $\Hom (H,A)$ by means of  
\[
(h \rhd \theta)(k) =  \wtil (k_{(1)},h_{(1)}) \theta(k_{(2)}h_{(2)}) \wtil^{-1}(k_{(3)},h_{(3)}).
\]
In fact, given $h,k \in H$ and $\theta, \phi \in \Hom(H,A)$, we have 
\begin{eqnarray*}
&&(h \rhd (\theta * \phi))(k) = \\
& = &  \wtil (k_{(1)},h_{(1)})(\theta * \phi)(k_{(2)}h_{(2)}) \wtil^{-1}(k_{(3)},h_{(3)})\\
& = &  \wtil (k_{(1)},h_{(1)})\theta (k_{(2)}h_{(2)}) \phi (k_{(3)}h_{(3)}) \wtil^{-1}(k_{(4)},h_{(4)})\\
& = &  \wtil (k_{(1)},h_{(1)})\theta (k_{(2)}h_{(2)}) \epsilon (k_{(3)}h_{(3)} ) \phi (k_{(4)}h_{(4)}) \wtil^{-1}(k_{(5)},h_{(5)})\\
& = &  [\wtil (k_{(1)},h_{(1)})\theta (k_{(2)}h_{(2)}) \wtil^{-1}(k_{(3)},h_{(3)} )]  
\times \\ 
& & \times [ \wtil(k_{(4)},h_{(4)}) \phi (k_{(5)}h_{(5)}) \wtil^{-1}(k_{(6)},h_{(6)})]\\
& = &  (h_{(1)} \rhd \theta)(k_{(1)}) (h_{(2)} \rhd \phi)(k_{(2)}) \\
& = &  (h_{(1)} \rhd \theta)*(h_{(2)} \rhd \phi)(k). 
\end{eqnarray*}

It is also easy to see that $1_H \rhd \theta =\theta$, for all 
$\theta \in \mbox{Hom}(H,A)$. Indeed, 
\begin{eqnarray*}
& & ( 1_H \cdot \theta )(h)= \wtil (h_{(1)} ,1_H ) \theta ( h_{(2)} 1_H ) \wtil^{-1} (h_{(3)} ,1_H ) \\
&=&  \epsilon (h_{(1)} ) \theta (h_{(2)} )\epsilon (h_{(3)}) =\theta (h).
\end{eqnarray*}

The map $\eta(k) = \epsilon(k) \um$ is the unit of $\Hom(H,A)$, and one needs to verify that $(h \rhd \eta)(k) = \epsilon (h) \eta (k)$. 
\begin{eqnarray*}
(h \rhd \eta)(k) & = & \wtil(k_{(1)},h_{(1)}) \eta( k_{(2)}h_{(2)})
\wtil^{-1}(k_{(3)},h_{(3)}) =\\
&=& \wtil(k_{(1)},h_{(1)}) \epsilon( k_{(2)}) \epsilon (h_{(2)})
\wtil^{-1}(k_{(3)},h_{(3)}) =\\
&=& \wtil( k_{(1)}, h_{(1)}) \wtil^{-1}(k_{(2)}, h_{(2)}) =\\
& = & \epsilon(h) \epsilon(k) \um = \epsilon (h) \eta(k).
\end{eqnarray*}

The twisted $H$-module structure appears when one expands the expression of $(h \rhd (k \rhd \theta))$:
\begin{eqnarray*}
& & (h \rhd (k \rhd \theta))(l) 
=  \wtil(l_{(1)},h_{(1)})\wtil(l_{(2)}h_{(2)},k_{(1)})\theta (l_{(3)}h_{(3)}k_{(2)})\wtil^{-1} (l_{(4)}h_{(4)},k_{(3)})  \\
&= &  \wtil(l_{(1)},h_{(1)})\wtil(l_{(2)}h_{(2)},k_{(1)}) 
\underbrace{ \epsilon(l_{(3)})\epsilon(h_{(3)}) \epsilon(k_{(2)})}
\theta (l_{(4)}h_{(4)}k_{(3)})
\underbrace{\epsilon(l_{(5)})\epsilon(h_{(5)}) \epsilon(k_{(4)}) }\times \\
& & \wtil^{-1} (l_{(6)}h_{(6)},k_{(5)}) \wtil^{-1}(l_{(7)},h_{(7)}) \\
&= &  \wtil(l_{(1)},h_{(1)})\wtil(l_{(2)}h_{(2)},k_{(1)})
\underbrace{\wtil^{-1}(l_{(3)},h_{(3)}k_{(2)})]   \wtil(l_{(4)},h_{(4)}k_{(3)})} 
 \theta (l_{(5)}h_{(5)}k_{(4)}) \times \\
&& \times 
\underbrace{ \wtil^{-1}(l_{(6)},h_{(6)}k_{(5)}) \wtil(l_{(7)},h_{(7)}k_{(6)})}
\wtil^{-1}(l_{(8)}h_{(8)},k_{(7)}) \wtil^{-1}(l_{(9)},h_{(9)}) \\
&= &  [\wtil(l_{(1)},h_{(1)})\wtil(l_{(2)}h_{(2)},k_{(1)})\wtil^{-1}(l_{(3)},h_{(3)}k_{(2)})] \times \\ 
& & \times \wtil(l_{(4)},h_{(4)}k_{(3)}) \theta (l_{(5)}h_{(5)}k_{(4)}) \wtil^{-1}(l_{(6)},h_{(6)}k_{(5)}) \times \\
& & \times [\wtil(l_{(7)},h_{(7)}k_{(6)}) \wtil^{-1}(l_{(8)}h_{(8)},k_{(7)}) \wtil^{-1}(l_{(9)},h_{(9)})] \\
&= &  [\wtil(l_{(1)},h_{(1)})\wtil(l_{(2)}h_{(2)},k_{(1)})\wtil^{-1}(l_{(3)},h_{(3)}k_{(2)})] (h_{(4)}k_{(3)} \rhd \theta (l_{(4)})) \\ 
&  &    [\wtil(l_{(5)},h_{(5)}k_{(4)}) \wtil^{-1}(l_{(6)}h_{(6)},k_{(5)}) \wtil^{-1}(l_{(7)},h_{(7)})] . \\
\end{eqnarray*}
If we define $u : H \otimes H \rightarrow \Hom(H,A)$ as the map 
\begin{equation}\label{def.u.globalization}
u(h,k)(l)=  \wtil(l_{(1)},h_{(1)}) \wtil(l_{(2)}h_{(2)},k_{ (1)}) \wtil^{-1}(l_{(3)},h_{(3)}k_{(2)}), 
\end{equation}
then its convolution inverse is 
\begin{equation} \label{u.inverse.globalization}
 u^{-1}(h,k)(l) =   \wtil(l_{(1)},h_{(1)}k_{(1)})\wtil^{-1}(l_{(2)}h_{(2)},k_{(2)})\wtil^{-1}(l_{(3)},h_{(3)}),
\end{equation}
and  we have 
\begin{equation}
\label{lawofcomposition}
(h \rhd (k \rhd \theta ))(l) =   u(h_{(1)},k_{(1)}) * (h_{(2)}k_{(2)} \rhd \theta ) * u^{-1} (h_{(3)},k_{(3)}) (l). 
\end{equation}

The map $u$ is truly a cocycle, i.e., it satisfies the cocycle identity
\begin{equation}
\label{lawofcocycles}
 (h_{(1)} \rhd u(k_{(1)},l_{(1)}))*u(h_{(2)},k_{(2)}l_{(2)})   =   u(h_{(1)},k_{(1)}) * u(h_{(2)}k_{(2)},l).
\end{equation}
In fact, beginning on the left hand side  and using repeatedly the definition of $u$, we get
\begin{eqnarray*}
& &  (h_{(1)} \rhd u(k_{(1)},l_{(1)}))*u(h_{(2)},k_{(2)}l_{(2)})(m) = \\
& = &  (h_{(1)} \rhd u(k_{(1)},l_{(1)}))(m_{(1)}) u(h_{(2)},k_{(2)}l_{(2)})(m_{(2)}) 
\\ 
& = &   
[\wtil(m_{(1)},h_{(1)}) u(k_{(1)},l_{(1)})(m_{(2)}h_{(2)}) \wtil^{-1}(m_{(3)},h_{(3)})] 
u(h_{  (4)  },k_{(2)}l_{(2)})(m_{(4)}) \\ 
& = &   \wtil(m_{(1)},h_{(1)}) u(k_{(1)},l_{(1)})(m_{(2)}h_{(2)}) \underbrace{\wtil^{-1}(m_{(3)},h_{(3)}) \wtil(m_{(4)},h_{(4)})} \times \\
& & \times 
\wtil(m_{(5)}h_{(5)},k_{(2)}l_{(2)}) \wtil^{-1}(m_{(6)},h_{(6)}k_{(3)}l_{(3)})\\ 
& = &   \wtil(m_{(1)},h_{(1)}) u(k_{(1)},l_{(1)})(m_{(2)}h_{(2)}) \wtil(m_{(3)}h_{(3)},k_{(2)}l_{(2)}) 
\wtil^{-1}(m_{(4)},h_{(4)}k_{(3)}l_{(3)})\\ 
& = &   \wtil(m_{(1)},h_{(1)}) \wtil(m_{(2)}h_{(2)},k_{(1)}) \wtil(m_{(3)}h_{(3)}k_{(2)},l_{(1)})
\times \\
& &  \times \underbrace{ \wtil^{-1}(m_{(4)}h_{(4)},k_{(3)}l_{(2)}) \wtil(m_{(5)}h_{(5)},k_{(4)}l_{(3)})} 
 \wtil^{-1}(m_{(6)},h_{(6)}k_{(5)}l_{(4)})\\ 
& = &  \wtil(m_{(1)},h_{(1)}) \wtil(m_{(2)}h_{(2)}, k_{(1)}) \wtil (m_{(3)}h_{(3)}k_{(2)},l_{(1)}) 
\wtil^{-1}(m_{(4)},h_{(4)}k_{ (3)  }l_{(2)})\\
& = &  \wtil(m_{(1)},h_{(1)}) \wtil(m_{(2)}h_{(2)}, k_{(1)}) \epsilon(m_{(3)}h_{(3)}k_{ (2)  }) \times \\
& & \times \wtil (m_{(4)}h_{(4)}k_{(3)},l_{(1)}) 
\wtil^{-1}(m_{(5)},h_{(5)}k_{(4)}l_{(2)})\\
& = &  [ \wtil(m_{(1)},h_{(1)}) \wtil(m_{(2)}h_{(2)}, k_{(1)}) \wtil^{-1}(m_{(3)},h_{(3)}k_{(2)})] \times \\
& & \times 
[\wtil(m_{(4)},h_{(4)}k_{(3)}) \wtil (m_{(5)}h_{(5)}k_{(4)},l_{(1)}) \wtil^{-1}(m_{(6)},h_{(6)}k_{(5)}l_{(2)})]\\
& = &   u(h_{(1)},k_{(1)}) (m_{(1)}) u(h_{(2)}k_{(2)},l)(m_{(2)})\\
& = &  u(h_{(1)},k_{(1)}) *u(h_{(2)}k_{(2)},l)(m).
\end{eqnarray*}

Finally, it is easily checked that $u$ is a normalized 2-cocycle:
if  $\eta $ denotes the unit element of $\Hom(H,A)$ then  
\begin{equation} \label{equation.u.is.normalized} \nonumber
u(h,1_H)=u(1_H,h) = \epsilon(h) \eta. 
\end{equation}

We will show now that $\varphi(A)$ is an ideal in a twisted $H$-module subalgebra of $\Hom(H,A)$.
Define $B$ as the subalgebra of $\Hom (H,A)$ generated by the elements of the form 
$h\triangleright \varphi (a)$, for all $h\in H$ and $a\in A$, 
and by the functions $u^{\pm 1} (h,k)$, for all $h,k\in H$. Since $u(1_H,1_H) (h) = \epsilon(1_H) \eta(h) = \eta(h)$, $B$ is a unital subalgebra of $\Hom(H,A)$. 
 
It is easy to see that $H\triangleright B \subseteq B$, because of the law of composition
(\ref{lawofcomposition}) and the cocycle identity (\ref{lawofcocycles}). Therefore, 
$B$ is a twisted $H$-module subalgebra of $\Hom (H,A)$. 
One needs only to verify that $\varphi (A) $ is an ideal in $ B.$ This is accomplished by showing the following identities:
\begin{enumerate}
\item[(i)] $\varphi (a) * (h\triangleright \varphi (b)) =\varphi (a(h\cdot b))$.
\item[(ii)] $(h\triangleright \varphi (b)) * \varphi (a) =\varphi ((h\cdot b)a)$.
\item[(iii)] $\varphi (a) * u^{\pm 1}(h,k) =\varphi (a \wtil^{\pm 1} (h,k))$.
\item[(iv)] $ u^{\pm 1}(h,k) *\varphi (a) =\varphi (\wtil^{\pm 1} (h,k) a)$.
\end{enumerate}

For identity (i) we have, for $a,b\in A$ and $h\in H$
\begin{eqnarray*}
& & \varphi (a) * (h \rhd \varphi (b)) (k) = \\
& = &  (k_{(1)} \cdot a) \wtil (k_{(2)},h_{(1)}) (k_{(3)} h_{(2)} \cdot b) 
\wtil^{-1}(k_{(4)},h_{(3)}) \\
& = &  (k_{(1)} \cdot a) (k_{(2)} \cdot \um ) \wtil (k_{(3)},h_{(1)}) 
(k_{(4)} h_{(2)} \cdot \um )(k_{(5)} h_{(3)} \cdot b) \times \nonumber\\
& & \times (k_{(6)} h_{(4)} \cdot \um ) \wtil^{-1}(k_{(7)},h_{(5)}) (k_{(8)} \cdot \um ),
\end{eqnarray*}
where the last factor appears because $e(k) = (k \cdot \um)$ commutes convolutionally. 
Since $f_1$ and $f_2$ are also central in $\Hom(H \otimes H, A)$ (by (STPA2) ) and  $\omega = f_1 * f_2 * \wtil$, the latter  expression can be rewritten as 
\begin{eqnarray*}
& = &  (k_{(1)} \cdot a) \omega (k_{(2)},h_{(1)}) (k_{(3)} h_{(2)} \cdot b) 
\omega '(k_{(4)},h_{(3)}) =\\
& = &  (k_{(1)} \cdot a) (k_{(2)} \cdot (h\cdot b))=\\
& = & k\cdot (a (h\cdot b) =\varphi (a(h\cdot b))(k). 
\end{eqnarray*}
Identity (ii) is obtained in a similar manner.

For  identity (iii), we have, for $a\in A$ and $h,k,l \in H$
\begin{eqnarray*}
& & \varphi (a)* u(h,k)(l) =  \varphi (a)(l_{(1)}) u(h,k)(l_{(2)}) =\\
& = &  (l_{(1)} \cdot a)\wtil (l_{(2)} , h_{(1)}) \wtil (l_{(3)} h_{(2)} ,k_{(1)}) \wtil^{-1} (l_{(4)}, h_{(3)}k_{(2)}) =\\
& = &  (l_{(1)} \cdot a) (l_{(2)} \cdot \um ) \wtil (l_{(3)} , h_{(1)}) \wtil (l_{(4)} h_{(2)} ,k_{(1)}) \wtil^{-1} (l_{(5)}, h_{(3)}k_{(2)}) =\\
& = &  (l_{(1)} \cdot a) (l_{(2)} \cdot \wtil (h_{(1)} , k_{(1)})) \wtil (l_{(3)} , h_{(2)} k_{(2)}) \wtil^{-1} (l_{(4)}, h_{(3)}k_{(3)}) =\\
& = &  (l_{(1)} \cdot a) (l_{(2)} \cdot \wtil (h , k))=\\
& = & l\cdot (a\wtil (h,k)) =\varphi (a\wtil (h,k))(l).
\end{eqnarray*}
In a similar way, we obtain the expression for $u^{-1} (h,k)$ and the identities (iv) as well. Therefore $\varphi (A) \trianglelefteq B$.

From  identity (i) we conclude quickly that the map $\varphi$ intertwines the partial action on $A$ with the partial action on $\varphi (A)$, indeed
\[
h\cdot \varphi (a) =\varphi (\um ) *(h\triangleright \varphi (a))=\varphi (h\cdot a) .
\]

The image of the partial cocycle $\varphi(\omega(h,k))$ is the induced cocycle 
\[
v(h,k) =  (h_{(1)} \cdot \varphi(\um))*u(h_{(2)},k_{(1)})*(h_{(3)}k_{(2)} \cdot \varphi(\um))
\]
for the twisted partial action of $H$ on the ideal $\varphi(A)$. In fact, 
\begin{eqnarray*}
& & \varphi( \omega (h,k))(l) = l \cdot \omega (h,k)  \\ 
& = &   l \cdot [(h_{(1)} \cdot \um)\omega(h_{(2)},k_{(1)})(h_{(3)}k_{(2)} \cdot \um)] \\
& = &  (l_{(1)} \cdot (h_{(1)} \cdot \um))(l_{(2)} \cdot \omega(h_{(2)},k_{(1)}))(l_{(3)} \cdot (h_{(3)}k_{(2)} \cdot \um)) \\
& = &  (l_{(1)} \cdot (h_{(1)} \cdot \um))(l_{(2)} \cdot [(h_{(2)} \cdot \um)\wtil(h_{(3)},k_{(1)})
(h_{(4)} k_{(2)} \cdot \um)]) \times \\
& & \times (l_{(3)} \cdot (h_{(5)}k_{(3)} \cdot \um)) \\
& = &  (l_{(1)} \cdot (h_{(1)} \cdot \um))(l_{(2)} \cdot \wtil(h_{(2)},k_{(1)}))(l_{(3)} \cdot (h_{(3)}k_{(2)} \cdot \um)).
 \end{eqnarray*}
Using equality (\ref{wtil}) for $\wtil$, we obtain
 \begin{eqnarray*}
&& \varphi(\omega(h,k))(l) =  l \cdot \omega(h,k)  = \\
& = &   (l_{(1)} \cdot (h_{(1)} \cdot \um))\wtil(l_{(2)},h_{(2)})\wtil(l_{(3)}h_{(3)},k_{(1)}) \times \\
& & \times(\wtil^{-1}(l_{(4)},h_{(4)}k_{(2)})) 
 (l_{(5)} \cdot (h_{(5)}k_{(3)} \cdot \um)) \\
& = &  (l_{(1)} \cdot (h_{(1)} \cdot \um )) u(h_{(2)},k_{(1)})(l_{(2)})
(l_{(3)} \cdot (h_{(3)}k_{(2)} \cdot \um)) \\
& = &  \varphi(h_{(1)} \cdot \um) *u(h_{(2)},k_{(1)})*\varphi(h_{(3)}k_{(2)} \cdot \um) (l)\\
& = &  (h_{(1)} \cdot \varphi(\um))*u(h_{(2)},k_{(1)})*(h_{(3)}k_{(2)} \cdot \varphi(\um)) (l) \\
& = & v(h,k) (l). 
\end{eqnarray*}

In an analogous manner, it can be shown that
\begin{eqnarray*}
\varphi(\omega'(h,k)) & = &  (h_{(1)}k_{(1)} \cdot \varphi(\um))*u^{-1}(h_{(2)},k_{(2)})*(h_{(3)} \cdot \varphi(\um)) \\
& = & v' (h,k )(l) .
\end{eqnarray*}
This concludes the proof.
\end{proof} 

For the case of $H$ being a co-commutative Hopf algebra, the $H$-submodule $H\triangleright \varphi (A)$ is a
non-unital subalgebra of $\Hom(H,A)$ and it carries a globalization of the twisted partial action on $A$.

This is because we can write, for $a, b\in A$ and $h,k \in H$, using the co-commutativity of $H$, the product $(h\triangleright \varphi (a))\ast (k\triangleright \varphi (b))$ as a linear combination in $H\triangleright \varphi (A)$. Indeed,
\begin{eqnarray*}
& \, & (h\triangleright \varphi (a))\ast (k\triangleright \varphi (b)) =(h_{(1)}\triangleright \varphi (a))\ast (h_{(2)}S(h_{(3)})\triangleright (k\triangleright \varphi (b))) \\
& = & (h_{(1)}\triangleright \varphi (a))\ast 
u^{-1} (h_{(2)}, S(h_{(7)})) \ast (h_{(3)} \triangleright ( S(h_{(6)})\triangleright 
(k\triangleright \varphi (b))))\ast u(h_{(4)}, S(h_{(5)}))
\\
& = & (h_{(1)}\triangleright \varphi (a))\ast 
u^{-1} (h_{(2)}, S(h_{(9)})) \ast (h_{(3)} \triangleright ( S(h_{(8)})\triangleright 
(k\triangleright \varphi (b))))\ast (u(h_{(4)}, S(h_{(7)})) \ast \\
&& \ast u(h_{(5)} S(h_{(6)}) , h_{(10)}) 
\\
& = & (h_{(1)}\triangleright \varphi (a))\ast 
u^{-1} (h_{(2)}, S(h_{(9)})) \ast (h_{(3)} \triangleright ( S(h_{(8)})\triangleright 
(k\triangleright \varphi (b))))\ast (h_{(4)} \triangleright u(S(h_{(7)}), h_{(10)})) \ast 
\\
&& \ast u(h_{(5)} , S(h_{(6)}) h_{(11)}) \\
& = & (h_{(1)}\triangleright \varphi (a))\ast 
u^{-1} (h_{(2)}, S(h_{(10)})) \ast (h_{(3)} \triangleright ( S(h_{(9)})\triangleright 
(k\triangleright \varphi (b))))\ast (h_{(4)} \triangleright u(S(h_{(8)}), h_{(11)})) \ast 
\\
&& \ast u(h_{(5)} , S(h_{(6)}) h_{(7)}) \\
& = & (h_{(1)}\triangleright \varphi (a))\ast 
u^{-1} (h_{(2)}, S(h_{(7)})) \ast (h_{(3)} \triangleright ( S(h_{(6)})\triangleright 
(k\triangleright \varphi (b))))\ast (h_{(4)} \triangleright u(S(h_{(5)}), h_{(8)}))   \\
& = & (h_{(1)}\triangleright \varphi (a))\ast 
u^{-1} (h_{(2)} S(h_{(3)}), h_{(4)}) \ast u^{-1} (h_{(5)}, S(h_{(9)})) \ast (h_{(6)} \triangleright ( S(h_{(6)})\triangleright 
(k\triangleright \varphi (b))))\ast 
\\
&& \ast (h_{(7)} \triangleright u(S(h_{(8)}), h_{(10)}))   \\
& = & (h_{(1)}\triangleright \varphi (a))\ast 
u^{-1} (h_{(2)} S(h_{(5)}), h_{(6)}) \ast u^{-1} (h_{(3)}, S(h_{(4)})) \ast (h_{(7)} \triangleright ( S(h_{(10)})\triangleright 
(k\triangleright \varphi (b))))\ast 
\\
&& \ast (h_{(8)} \triangleright u(S(h_{(9)}), h_{(11)}))   \\
& = & (h_{(1)}\triangleright \varphi (a))\ast 
u^{-1} (h_{(2)} , S(h_{(5)}) h_{(6)}) \ast (h_{(3)} \triangleright u^{-1} (S(h_{(4)}), h_{(7)})) \ast (h_{(8)} \triangleright ( S(h_{(11)})\triangleright 
(k\triangleright \varphi (b))))\ast 
\\
&& \ast (h_{(9)} \triangleright u(S(h_{(10)}), h_{(12)}))   \\
& = & (h_{(1)}\triangleright \varphi (a))\ast 
(h_{(2)} \triangleright u^{-1} (S(h_{(6)}), h_{(7)})) \ast 
(h_{(3)} \triangleright ( S(h_{(5)})\triangleright (k\triangleright \varphi (b))))
\ast (h_{(3)} \triangleright u(S(h_{(4)}), h_{(8)}))   \\
& = & h_{(1)}\triangleright (\varphi (a)\ast 
 u^{-1} (S(h_{(4)}), h_{(5)}) \ast 
(S(h_{(3)})\triangleright (k\triangleright \varphi (b)))
\ast  u(S(h_{(2)}), h_{(6)})) ,  
\end{eqnarray*}
and this last expression lies in $H\triangleright \varphi (A)$, because $\varphi (A)$ is an ideal of the algebra $B$ obtained in the previous theorem.

And, again only for co-commutative Hopf algebras, given any $h,k,l\in H$, the following expression in $B$
\[
\theta (h,k,l)=  u^{-1} (S(h_{(3)}), h_{(4)}) *(S(h_{(2)})\triangleright u(k,l)) * 
u(S(h_{(1)}), h_{(5)}) 
\]
satisfies the relation
\[
h_{(1)} \triangleright \theta (h_{(2)} ,k,l)=\epsilon (h) u(k,l).
\]
Then, since $ \theta (h,k,l) \in B$ and $ \varphi (A)$ is an ideal in $B,$ we have,  for any $a\in A$, and $h,k,l\in H,$ that
\[
\varphi (a) * \theta (h,k,l) \in \varphi (A),
\]
which implies that
\begin{eqnarray*}
h_{(1)}\triangleright (\varphi (a) * \theta (h_{(2)}, k,l)) & = &(h_{(1)}\triangleright \varphi (a)) * (h_{(2)}\triangleright \theta (h_{(3)},k,l) )  =\\
& = & (h_{(1)}\triangleright \varphi (a)) *\epsilon (h_{(2)}) u(k,l) =\\
& = & (h\triangleright \varphi (a)) *u(k,l).
\end{eqnarray*}
That is, the product 
$(h\triangleright \varphi (a))*u(k,l)\in H\triangleright \varphi (A)$. Analogously, we obtain 
\[
 h_{(1)}\triangleright ( \theta (h_{(2)}, k,l) * \varphi (a)) = 
  u(k,l)*(h\triangleright \varphi (a)).
\]
Note that $H \triangleright \varphi(A)$ is a non-unital algebra. The above computation shows that  
for any $k,l$ in $H$, the elements $u^{\pm 1}(k,l)$ can be viewed as multipliers of the algebra $ H \triangleright \varphi(A)$ since
any product of elements of the form $h\triangleright \varphi (a)$ and $u(k,l)$ can be written as an element of $H\triangleright \varphi (A)$.
This is the case for globalization of twisted partial group actions as shown in \cite{DES2}. 
In this case, for each $g\in G$, the algebra which carries the globalization is $B= \sum_{g \in G} \beta_g (\varphi(A))$, the subspaces
$\beta_g(\varphi (A))$ are ideals of this algebra, 
and the 2-cocycle components $u_{g,h}$ belong to the multiplier algebra of $B$.

\begin{ex} \label{exemplo:grupos_algebricos}\cite{ABDP} Let $\kappa $ be an isomorphic copy of the complex numbers $\Cc $ and let $\mathbb{S}^1  \subseteq \Cc$ be the circle group, i.e. the group of all complex roots of $1$.  Let, furthermore, $G$ be an arbitrary finite group seen as a subgroup of $S_n$ for some $n$. Taking the action of  $G \subseteq S_n$  on $ (\ka \mathbb{S}^1 )^{\otimes n} $  by permutation of roots,  consider  the smash product Hopf algebra 
\[
H'_1 = (\ka \mathbb{S}^1 )^{\otimes n} \rtimes \ka  G, 
\]
which is co-commutative. Let $X\subseteq G$ be an arbitrary subset which is not a subgroup, and consider the subalgebra $\tilde {A} =  (\sum _{g\in X} p_g) (\ka  G)^* \subseteq  (\ka  G)^*,$ and define the commutative algebra $A'=  \ka [t,t^{-1}]^{\otimes n} \otimes \tilde{A}$. 

In order to  simplify the notation, write 
\begin{equation}\label{notation1}
t_i = 1 \otimes \ldots \otimes 1 \otimes t \otimes 1\otimes \ldots \otimes 1,
\end{equation} 
where $t$ belongs to the $i$-copy of $\ka [ t, t\m]$, then, we have the elementary monomials in $\ka [ t, t\m]$, given by
\[
t_1^{k_1}\ldots t_n^{k_n} = t^{k_1} \otimes \ldots \otimes t^{k_n} .
\]
In its turn, the generators of $(\ka \mathbb{S}^1 )^{\otimes n}$ can be written in terms of the roots  $\chi_{\theta}$ of unity in the following way
\begin{equation}\label{notation2}  
\chi_{\theta_1 , \ldots \theta_n } = \chi_{\theta_1} \otimes \ldots \otimes \chi_{\theta_n } \in (\ka \mathbb{S}^1 )^{\otimes n},
\end{equation} 
where $\chi_{\theta_i} \in \mathbb{S}^1 $ is the root of $1$ whose angular coordinate is $\theta _i $ and which belongs to the  $i$-factor of $(\ka \mathbb{S}^1 )^{\otimes n}$.

Then with the notation established in (\ref{notation1}) and (\ref{notation2}), the formula 

$$
(\chi_{\theta_1 , \ldots \theta_n } \otimes u_g) \cdot   (t_1^{k_1} \ldots t_n^{k_n} \otimes p_s )  =    
{ \left \{  \begin{array}{ll}
\exp\{ i \sum_{j=1}^{n}k_j \theta_{  g s \m  (j)}  \}  \; t_1^{k_1} \ldots t_n^{k_n} \otimes p_{ s g \m    },
 & \mbox{\rm if \ } s\m g \in X, \\
0 & \mbox{\rm if \ } s\m g \not\in X,
\end{array}  \right. }
$$ where $g \in G$ and $s \in X \subseteq G,$   gives a left partial action  $\cdot : H'_1 \times A' \to A' $. As $H'_1 $ is a co-commutative Hopf algebra and $A'$ is a commutative algebra, then thanks to Remark~\ref{rem:co-commutativeCase} it is enough to give a partial $2$-cocycle $\omega $ attached to our partial action, which is invertible in the ideal $\langle f_1 \ast f_2\rangle ,$ i.e. a $\kappa $-linear map   $\omega:H\otimes H \to A$ satisfying  conditions (TPA3)-(TPA5) and (STPA3).

 Assume   that  $G$ is such that its  Schur Multiplier  is non-trivial, and take a normalized $2$-cocycle $\gamma : G \times G \to {\ka }^{\ast}$ which is not a coboundary (for a concrete such $\gamma $ when $G$ is the Klein four-group see \cite{ABDP}). For arbitrary  
$h= \chi_{\theta_1 , \ldots \theta_n } \otimes u_g$ and $l= \chi_{{\theta}'_1 , \ldots {\theta }'_n } \otimes u_s$ in $H'_1$  set
\begin{equation*}
\omega ( h, l ) = \gamma (g,s)\;  ( h \cdot ( l \cdot {\mathbf{1}_{A'} } )).
\end{equation*}   
Then, as it was explained in \cite{ABDP}, the pair $(\alpha , \omega )$ forms a twisted partial action of 
$H'_1 = (\ka \mathbb{S}^1 )^{\otimes n} \rtimes \ka  G$ on $A'$. Moreover, taking 
\begin{equation*}
\omega '( h, l ) = \gamma (g,s)\m \;  ( h \cdot ( l \cdot {\mathbf{1}_{A'} } )),
\end{equation*}   we readily see that it is symmetric.

We   observe now that this symmetric twisted partial action is globalizable.  Indeed, the auxiliary convolution invertible map $\wtil :H'_1 \otimes H'_1 \rightarrow A'$ can be given by
\begin{equation*}
\wtil ( h, l ) = \gamma (g,s) \epsilon (h) \epsilon (l),
\end{equation*} 
where $h= \chi_{\theta_1 , \ldots \theta_n } \otimes u_g$ and $l= \chi_{{\theta}'_1 , \ldots {\theta }'_n } \otimes u_s$.  Direct verifications show  that the convolution inverse of $\wtil $ is 
\begin{equation*}
\wtil^{-1} ( h, l ) = \gamma (g,s)^{-1} \epsilon (h) \epsilon (l),
\end{equation*} and $\wtil ,$ $\wtil \m$ satisfy all conditions of Theorem~\ref{teorema:globalizacao}. 
\end{ex}

{ 

\begin{ex} Consider now  the symmetric twisted partial action given in Example \ref{exemplo:klein4}, where $H=(\kappa K_4)^*$, the dual of the group algebra of the  Klein four-group $K_4 =\langle a,b \; | \; a^2 =b^2 =e \rangle$, and $A$ is the base field $\kappa$, defined by the measuring map $\lambda (p_e ) =\lambda (p_a )=\frac{1}{2}$ and $\lambda (p_b )=\lambda (p_{ab} )=0$. As it was shown in this case,  the partial 2-cocyle $\omega $ and its partial inverse ${\omega}'$ are fully determined by a pair $(x,y)\in \kappa^2$ satisfying the quadratic equation 
(\ref{points}) 
where   $x=\omega (p_e ,p_e )$ and    $y=\omega' (p_e ,p_e )$.  One readily notes that the pairs $(\frac{1}{4}, \frac{1}{4}),$  $(\frac{1}{8}, \frac{1}{8})$ are  solutions for (\ref{points}). Obviously, the first one gives  the trivial partial $2$-cocycle 
$\omega (p_g, p_h) = \omega ' (p_g, p_h) = f_1 \ast f_2 \; (p_g, p_h) =p_g \cdot (p_h \cdot 1).$ So we take the second one, and then we have  
\[
x = \omega (p_e ,p_e )=\omega (p_a , p_e ) =\omega (p_e , p_a ) =\omega (p_a , p_a )  =  \frac{1}{8}
\]
and
\begin{eqnarray*}
&\, & \omega (p_e ,p_b )=\omega (p_a , p_b ) =\omega (p_e , p_{ab} ) =\omega (p_a , p_{ab} )  \\
& = & \omega (p_b ,p_e )=\omega (p_{ab} , p_e ) =\omega (p_b , p_a ) =\omega (p_{ab} , p_a )  \\
& = & -\omega (p_b ,p_b )=-\omega (p_{ab} , p_b ) =-\omega (p_b , p_{ab} ) =-\omega (p_{ab} , p_{ab} ) \\
& = & - \frac{1}{8},
\end{eqnarray*} and $\omega = \omega ' .$

 We shall prove that  this symmetric twisted partial action is globalizable. 
In order to  consider the auxiliary map $\wtil$, let us denote $X_{g,h} =\wtil (p_g ,p_h)$, for $g,h\in K_4,$ so that $\wtil $  is determined by $16$ variables.  The normalization condition gives us
\begin{equation} \label{equation1}
\sum_{g\in K_4} \wtil (p_g ,p_h )= \sum_{g\in K_4} X_{g,h} =\epsilon (p_h )=\left\{ \begin{array}{lcr} 1 & \mbox{ if } & h=e , \\ 0 & \mbox{ if } & h\neq e . \end{array} \right. 
\end{equation} and 
\begin{equation} \label{equation1.2}
\sum_{h\in K_4} \wtil (p_g ,p_h )= \sum_{h\in K_4} X_{g,h} =\epsilon (p_g )=\left\{ \begin{array}{lcr} 1 & \mbox{ if } & g=e , \\ 0 & \mbox{ if } & g\neq e . \end{array} \right. 
\end{equation}
From  condition  (\ref{wtil}), we have the following equations, for $g,h, f \in K_4$,
\begin{equation} 
\sum_{\substack{r\in \langle a \rangle \\
				  s\in K_4}}
X_{s,f h\m s} \ X_ {r\m g , s\m h} =
\sum_{\substack{ x\in \langle a \rangle \\
				  y\in K_4}} 
		X_{x\m y, h g\m y }\  X_ {y\m g , f }
\end{equation}
whereas from (\ref{globcocycle1}), we obtain 
\begin{equation} \label{equation2}
\omega (p_g ,p_h ) =\sum_{r,s\in K_4 } \lambda (p_r )\lambda (p_s ) \wtil (p_{r^{-1}g} , p_{s^{-1}h})=\frac{1}{4} \sum_{r,s\in \langle a \rangle } X_{r^{-1}g , s^{-1}h}, 
\end{equation}  with  $g,h\in K_4 .$
Next, writing  $Y_{g,h} =\wtil^{-1} (p_g , p_h )$ we have also  the equations
\begin{equation} \label{equation3}
(\wtil   \ast \wtil^{-1})  (p_g , p_h ) = \sum_{r,s\in K_4} X_{r,s} Y_{r^{-1}g ,s^{-1}h} =\epsilon (p_g) \epsilon (p_h) =\left\{ \begin{array}{lcr} 1 & \mbox{ if } & (g,h)=(e,e) , \\ 0 & \mbox{ if } & (g,h)\neq (e,e), \end{array} \right. 
\end{equation} as well as 
\begin{equation} \label{equation2.1}
\omega ' (p_g ,p_h ) =\sum_{r,s\in K_4 } \lambda (p_r )\lambda (p_s ) \wtil \m (p_{r^{-1}g} , p_{s^{-1}h})=\frac{1}{4} \sum_{r,s\in \langle a \rangle } Y_{r^{-1}g , s^{-1}h}, 
\end{equation}  with  $g,h\in K_4 ,$ the latter coming from (\ref{globcocycle2}). Direct computations show that 
$$X_{1,1} = Y_{1,1} = X_{1,b}=Y_{1,b} = X_{b,1}=Y_{b,1} = -X_{b,b}=-Y_{b,b} =\frac{1}{2},$$  $$X_{ab^i,b^j} = Y_{ab^i,b^j} = X_{a^i,ab^j} = Y_{b^i,ab^j} =X_{ab^i,ab^j} = Y_{ab^i,ab^j} =0,$$ $i,j=0,1,$ is a solution for  (\ref{equation1})-(\ref{equation2.1}), showing that our symmetric twisted partial action is globalizable.
\end{ex}

}


\end{document}